\documentclass[12pt,twoside,a4paper]{article}
\usepackage{amsmath,amsthm, amssymb, amsfonts,latexsym,amscd}
\usepackage[dvips]{graphicx}
\newtheorem{theorem}{Theorem}[section]
\newtheorem{corollary}[theorem]{Corollary}
\newtheorem{lemma}[theorem]{Lemma}
\newtheorem{proposition}[theorem]{Proposition}

\begin{document}

\title{\bf Vertex connectivity of the power graph of a finite cyclic group II}
\author{Sriparna Chattopadhyay\footnote{Supported by SERB NPDF scheme (File No. PDF/2017/000908), Department of Science and Technology, Government of India} \and Kamal Lochan Patra \and Binod Kumar Sahoo}
\date{}
\maketitle

\begin{abstract}
The power graph $\mathcal{P}(G)$ of a given finite group $G$ is the simple undirected graph whose vertices are the elements of $G$, in which two distinct vertices are adjacent if and only if one of them can be obtained as an integral power of the other. The vertex connectivity $\kappa(\mathcal{P}(G))$ of $\mathcal{P}(G)$ is the minimum number of vertices which need to be removed from $G$ so that the induced subgraph of $\mathcal{P}(G)$ on the remaining vertices is disconnected or has only one vertex. For a positive integer $n$, let $C_n$ be the cyclic group of order $n$. Suppose that the prime power decomposition of $n$ is given by $n =p_1^{n_1}p_2^{n_2}\cdots p_r^{n_r}$, where $r\geq 1$, $n_1,n_2,\ldots, n_r$ are positive integers and $p_1,p_2,\ldots,p_r$ are prime numbers with $p_1<p_2<\cdots <p_r$. The vertex connectivity $\kappa(\mathcal{P}(C_n))$ of $\mathcal{P}(C_n)$ is known for $r\leq 3$, see \cite{panda, cps}. In this paper, for $r\geq 4$, we give a new upper bound for $\kappa(\mathcal{P}(C_n))$ and determine $\kappa(\mathcal{P}(C_n))$ when $n_r\geq 2$. We also determine $\kappa(\mathcal{P}(C_n))$ when $n$ is a product of distinct prime numbers.\\

\end{abstract}

\section{Introduction}

Let $\Gamma$ be a simple graph with vertex set $V.$ A subset $X$ of $V$ is called a (vertex) {\it cut-set} of $\Gamma$ if the induced subgraph of $\Gamma$ with vertex set $V\setminus X$ is  disconnected. A cut-set $X$ of $\Gamma$ is called a {\it minimal cut-set} if $X\setminus \{x\}$ is not a cut-set of $\Gamma$ for any $x\in X$. If $X$ is a minimal cut-set of $\Gamma$, then any proper subset of $X$ is not a cut-set of $\Gamma$. A cut-set $X$ of $\Gamma$ is called a {\it minimum cut-set} if $|X|\leq |Y|$ for any cut-set $Y$ of $\Gamma$. Clearly, every minimum cut-set of $\Gamma$ is also a minimal cut-set. The {\it vertex connectivity} of $\Gamma$, denoted by $\kappa(\Gamma)$, is the minimum number of vertices which need to be removed from $V$ so that the induced subgraph of $\Gamma$ on the remaining vertices is disconnected or has only one vertex. The latter case arises only when $\Gamma$ is a complete graph. If $\Gamma$ is not a complete graph and $X$ is a minimum cut-set of $\Gamma$, then $\kappa(\Gamma)=|X|$. A {\it separation} of $\Gamma$ is a pair $(A,B)$, where $A,B$ are disjoint non-empty subsets of $V$ whose union is $V$ and there is no edge of $\Gamma$ containing vertices from both $A$ and $B$. Thus, $\Gamma$ is disconnected if and only if there exists a separation of it.

\subsection{Power graph}

The notion of directed power graph of a group was introduced in \cite{ker1}, which was further extended to semigroups in \cite{ker1-1, ker2}. Then the notion of undirected power graph of a semigroup, in particular, of a group was defined in \cite{ivy}. Many researchers have investigated both the directed and undirected power graphs of groups from different view points. More on these graphs can be found in the survey paper \cite{survey} and the references therein.

Let $G$ be a finite group. The {\it power graph} $\mathcal{P}(G)$ of $G$ is the simple undirected graph with vertex set $G$, in which two distinct vertices are adjacent if and only if one of them can be obtained as an integral  power of the other. Thus two distinct vertices $x,y\in G$ are adjacent in $\mathcal{P}(G)$ if and only if $x\in\langle y\rangle$ or $y\in\langle x\rangle$. Since $G$ is finite, the identity element of $G$ is adjacent to all other vertices and so $\mathcal{P}(G)$ is connected.

The automorphism group of the power graph of a finite group was described in \cite[Theorem 2.2]{FMW-1}. Clearly, if two finite groups are isomorphic, then they have isomorphic power graphs. The converse statement does not hold in general: two non-isomorphic finite $p$-groups of the same order and each of exponent $p$ have isomorphic power graphs. However, by \cite[Corollary 3]{cam}, two finite groups with isomorphic power graphs have the same number of elements of each order. It was proved in \cite[Theorem 1]{cam-1} that two finite abelian groups are isomorphic if their power graphs are isomorphic. If $G$ and $H$ are two finite groups with isomorphic power graphs, where $H$ is a simple group, a cyclic group, a symmetric group, a dihedral group or a generalized quaternion (dicyclic) group, then $G$ is isomorphic to $H$ \cite[Theorem 15]{mir}.
The power graph of a finite group is complete if and only if the group is cyclic of prime power order \cite[Theorem 2.12]{ivy}. It was proved in \cite[Theorem 1.3]{cur} and \cite[Corollary 3.4]{cur-1} that, among all finite groups of a given order, the cyclic group of that order has the maximum number of edges and has the largest clique in its power graph. By \cite[Theorem 5]{dooser} and \cite[Corollary 2.5]{FMW}, the power graph of a finite group is perfect, in particular, the clique number and the chromatic number coincide. Explicit formula for the clique number of the power graph of a finite cyclic group is given in \cite[Theorem 2]{mir} and \cite[Theorem 7]{dooser}.

For a subset $A$ of $G$, we denote by $\mathcal{P}(A)$ the induced subgraph of $\mathcal{P}(G)$ with vertex set $A$. The subgraph $\mathcal{P}^{\ast}(G)=\mathcal{P}(G\setminus\{1\})$ of $\mathcal{P}(G)$ is called the {\it proper power graph} of $G$. The finite groups for which the proper power graph is strongly regular (respectively, planar, bipartite) are characterized in \cite{mog}. In the same paper, the authors proved connectedness of the proper power graph of certain groups. For the dihedral group $D_{2n}$ of order $2n$, the identity element is a cut-vertex of $\mathcal{P}(D_{2n})$ and so $\mathcal{P}^{\ast}(D_{2n})$ is disconnected. If $G$ is one of the groups $PGL(2,p^n)$ ($p$ an odd prime), $PSL(2,p^n)$ ($p$ prime), or a Suzuki group $Sz(2^{2n+1})$, then $\mathcal{P}^{\ast}(G)$ is disconnected \cite[Theorems 3.5--3.7]{doos}. In \cite[Section 4]{doos}, the authors proved that $\mathcal{P}^{\ast}(S_n)$ and $\mathcal{P}^{\ast}(A_n)$ are disconnected for many values of $n$, where $S_n, A_n$ are the symmetric and alternating groups respectively. The number of connected components of $\mathcal{P}^{\ast}(S_n)$ and $\mathcal{P}^{\ast}(A_n)$ are studied in \cite{BIS, BIS-1}.

\subsection{Vertex connectivity}

For a given finite group, determining the vertex connectivity of its power graph is an interesting problem. Clearly, every cut-set of the power graph contains the identity element of the group. The vertex connectivity of the power graph is $1$ if and only if the group is of order $2$ or its proper power graph is disconnected. We recall a few results on the vertex connectivity of the power graph of finite $p$-groups and cyclic groups. If $G$ is a cyclic $p$-group, then $\mathcal{P}(G)$ is a complete graph and so $\kappa\left(\mathcal{P}\left(G\right)\right)=|G|-1$. If $G$ is a generalized quaternion $2$-group (in general, a dicyclic group), then the set consisting of the identity element and the unique involution of $G$ is a minimum cut-set of $\mathcal{P}(G)$ and so $\kappa\left(\mathcal{P}\left(G\right)\right)=2$ \cite[Theorem 7]{sri}. If $G$ is a finite $p$-group, then $\mathcal{P}^{\ast}(G)$ is connected if and only if $G$ is either cyclic or a generalized quaternion $2$-group by \cite[Corollary 4.1]{mog} (also see \cite[Theorem 2.6 (1)]{doos}). In particular, $\kappa(\mathcal{P}(G))=1$ if $G$ is a finite non-cyclic abelian $p$-group, in this case the number of connected components of $\mathcal{P}^{\ast}(G)$ is obtained in \cite[Theorem 3.3]{panda}.

For a given positive integer $n$, let $C_n$ denote the finite cyclic group of order $n$. The number of generators of $C_n$ is $\phi(n)$, where $\phi$ is the Euler's totient function. We assume that $n$ is divisible by at least two distinct primes. The identity element and the generators of $C_n$ are adjacent to all other vertices of $\mathcal{P}(C_n)$. So every cut-set of $\mathcal{P}(C_n)$ must contain these elements, giving
$\kappa(\mathcal{P}(C_n))\geq \phi(n) +1$. Further, equality holds if and only if $n$ is a product of two distinct primes, see \cite[Proposition 2.5]{panda} and \cite[Lemma 2.5]{cps}.
For $n=p_1^{n_1}p_2^{n_2}$, where $p_1, p_2$ are distinct primes and $n_1,n_2$ are positive integers, it was proved in \cite[Theorem 2.7]{CP-2015} that $\kappa(\mathcal{P}(C_n))\leq \phi(n)+p_1^{n_1 -1}p_2^{n_2 -1}$. If $n=p_1p_2p_3$ is a product of three primes with $p_1<p_2<p_3$, then $\kappa(\mathcal{P}(C_n))\leq \phi (n)+ p_1+p_2 - 1$ by \cite[Theorem 2.9]{CP-2015}. These results were generalized in \cite{cps} and \cite{panda}.

Let $n=p_1^{n_1}p_2^{n_2}\cdots p_r^{n_r}$, where $r\geq 2$, $n_1,n_2,\ldots, n_r$ are positive integers and $p_1,p_2,\ldots,p_r$ are prime numbers with $p_1<p_2<\cdots <p_r$. Consider the integers $\alpha(n)$ and $\beta(n)$, where
\begin{align*}
\alpha(n)& :=\phi(n) + \frac{n}{p_1p_2 \cdots p_r}\times \left[p_1p_2\cdots p_{r-1}-\phi\left(p_1p_2\cdots p_{r-1}\right)\right],\\
\beta(n) & :=\phi(n) + \frac{n}{p_1p_2\cdots p_r}\times \frac{1}{p_r^{n_r-1}}\left[p_1p_2\cdots p_{r-1} +\phi\left(p_1p_2\cdots p_{r-1}\right)\left( p_r^{n_r-1}-2\right) \right].
\end{align*}
By \cite[Theorems 2.23, 2.35, 2.36]{panda},
\begin{equation}\label{bound-1}
\kappa(\mathcal{P}(C_n))\leq \alpha(n),
\end{equation}
\begin{equation}\label{bound-2}
\kappa(\mathcal{P}(C_n))\leq \beta(n)
\end{equation}
and the following hold:
\begin{enumerate}
\item[(i)] $\alpha(n) < \beta(n)$ if and only if $n_r\geq 2$ and $2\phi(p_1p_2\cdots p_{r-1}) > p_1p_2\cdots p_{r-1}$,
\item[(ii)] $\alpha(n) > \beta(n)$ if and only if $n_r\geq 2$ and $2\phi(p_1p_2\cdots p_{r-1}) < p_1p_2\cdots p_{r-1}$,
\item[(iii)] $\alpha(n)=\beta(n)$ if and only if $n_r=1$ or $(r,p_1)=(2,2)$.
\end{enumerate}

The authors of the present paper also independently obtained both the upper bounds (\ref{bound-1}) and (\ref{bound-2}) in \cite{cps}. Moreover, it was proved that if $2\phi(p_1p_2\cdots p_{r-1}) \geq p_1p_2\cdots p_{r-1}$, then the bound (\ref{bound-1}) is sharp, that is, $\kappa(\mathcal{P}(C_n))= \alpha(n)$ \cite[Theorem 1.3(i),(iii)]{cps}. As a consequence, if $p_1\geq r$, then $\kappa(\mathcal{P}(C_n))= \alpha(n)$ \cite[Corollary 1.4]{cps}. In particular, if $n=p_1^{n_1}p_2^{n_2}$, then $\kappa(\mathcal{P}(C_n))=\phi \left(p_{1}^{n_{1}}p_{2}^{n_2}\right)+p_{1}^{n_{1}-1}p_{2}^{n_2-1}$, also see \cite[Theorem 2.38]{panda}. It was shown in \cite[Theorem 1.5]{cps} that the bound (\ref{bound-2}) is sharp, that is, $\kappa(\mathcal{P}(C_n))= \beta(n)$ for integers $n=p_1^{n_1}p_2^{n_2}p_3^{n_3}$ with $2\phi(p_1p_2) < p_1p_2$ (so necessarily $p_1=2$). However, by \cite[Example 3.4]{cps}, equality may not hold in (\ref{bound-2}) in general if $2\phi(p_1p_2\cdots p_{r-1}) < p_1p_2\cdots p_{r-1}$. In fact, the present paper is an outcome of the study of the behaviour of this example.

In view of the results mentioned in the previous paragraphs, the vertex connectivity of $\mathcal{P}(C_n)$ is completely determined for $r\leq 3$. Define the following integer:
$$\gamma(n):= \phi(n)+\frac{n}{p_1\cdots p_r}\times \left[\phi(p_1\cdots p_{r-1})+ \phi(p_1\cdots p_{r-2}p_r)+p_1\cdots p_{r-2}-\phi(p_1\cdots p_{r-2})\right].$$
We prove the following three results in this paper.

\begin{theorem}\label{res-1}
Let $n=p_1^{n_1}p_2^{n_2}\cdots p_r^{n_r}$, where $r\geq 3$, $n_1,n_2,\ldots, n_r$ are positive integers and $p_1,p_2,\ldots,p_r$ are prime numbers with $p_1<p_2<\cdots <p_r$.
Then $\kappa(\mathcal{P}(C_n))\leq \gamma(n)$.
\end{theorem}

\begin{theorem}\label{res-2}
Let $n=p_1^{n_1}p_2^{n_2}\cdots p_r^{n_r}$, where $r\geq 3$, $n_1,n_2,\ldots, n_r$ are positive integers and $p_1,p_2,\ldots,p_r$ are prime numbers with $p_1<p_2<\cdots <p_r$. If $n_r\geq 2$, then
$$\kappa(\mathcal{P}(C_n)) = \min \{\alpha(n), \beta(n)\}.$$
\end{theorem}

\begin{theorem}\label{res-3}
Let $n=p_1p_2\cdots p_r$, where $r\geq 3$ and $p_1,p_2,\ldots,p_r$ are prime numbers with $p_1<p_2<\cdots <p_r$. Then
$\kappa(\mathcal{P}(C_n))= \min \{\alpha(n),\gamma(n)\}$.
\end{theorem}

\section{Preliminaries}

Recall that $\phi$ is a multiplicative function, that is, $\phi(ab)=\phi(a)\phi(b)$ for any two positive integers $a,b$ which are relatively prime. We have $\phi(p^k)=p^{k-1}(p-1)=p^{k-1}\phi(p)$ for any prime $p$ and positive integer $k$. Also, $\underset{d\mid m}{\sum} \phi(d) = m$ for every positive integer $m$.

For an element $x\in C_n$, we denote by $o(x)$ the order of $x$. Let $x,y$ be two distinct elements of $C_n$. If $x,y$ are adjacent in $\mathcal{P}(C_n)$, then $o(x)\mid o(y)$ or $o(y)\mid o(x)$ according as $x\in\langle y\rangle$ or $y\in\langle x\rangle$. The converse statement is also true, that is, if $o(x)\mid o(y)$ or $o(y)\mid o(x)$, then $x,y$ are adjacent in $\mathcal{P}(C_n)$. This follows from the fact that $C_n$ (being cyclic) has a unique subgroup of order $d$ for every positive divisor $d$ of $n$. We shall use the converse statement frequently without mentioning it.
For a positive divisor $d$ of $n$, define the following two sets:
\begin{enumerate}
\item[] $E_d:=\{x\in C_n: o(x)=d\}$, the set of all elements of $C_n$ whose order is $d$,
\item[] $S_d:=\{x\in C_n: o(x)\mid d\}$, the set of all elements of $C_n$ whose order divides $d$.
\end{enumerate}
Then $S_d$ is a cyclic subgroup of $C_n$ of order $d$ and $E_d$ is precisely the set of generators of $S_d$. So
$\left\vert S_d\right\vert =d$ and $\left\vert E_d\right\vert =\phi(d)$. Note that any cut-set of $\mathcal{P}(C_n)$ must contain the two sets $E_n$ and $E_1$, as each element from these two sets is adjacent with all other elements.

For a given non-empty proper subset $X$ of $C_n$, we define $\overline{X}:=C_n\setminus X$ and denote by $\mathcal{P}(\overline{X})$ the induced subgraph of $\mathcal{P}(C_n)$ with vertex set $\overline{X}$. The following result is very useful throughout the paper, see \cite[Lemma 2.1]{cps-1}.

\begin{lemma}\cite{cps-1}\label{simple}
If $X$ is a minimal cut-set of $\mathcal{P}(C_n)$, then either $E_d\subseteq X$ or $E_d\cap X=\emptyset$ for each positive divisor $d$ of $n$.
\end{lemma}

\noindent As a consequence of Lemma \ref{simple}, we have

\begin{corollary}
Suppose that $X$ is a minimal cut-set of $\mathcal{P}(C_n)$ and $A\cup B$ is a separation of $\mathcal{P}(\overline{X})$. Then for every positive divisor $d$ of $n$, there are three possibilities for the set $E_d$: either $E_d\subseteq X$, $E_d\subseteq A$ or $E_d\subseteq B$.
\end{corollary}

\subsection{Elementary results}

The following result can be found in \cite[Lemma 3.1]{cps-1}.

\begin{lemma}\label{inequ-1}
Let $p_1<p_2< \ldots <p_t$ be prime numbers with $t\geq 1$. Then $q \phi(p_1p_2\cdots p_{t}) \geq p_1p_2\cdots p_{t}$ for any integer $q\geq t+1$, with equality when $(t,p_1,q)=(1,2,2)$ or $(t,p_1,p_2,q)=(2,2,3,3)$.
\end{lemma}

\noindent The proof of the following lemma is similar to that of \cite[Lemma 2.1]{cps}.

\begin{lemma}\label{inequ-2}
Let $m=p_1^{m_1}p_2^{m_2}\cdots p_t^{m_t}$, where $t\geq 2$, $m_1,m_2,\ldots, m_t$ are positive integers and $p_1,p_2,\ldots,p_t$ are prime numbers with $p_1<p_2<\cdots <p_t$. Then
$$\phi\left(\frac{m}{p_i}\right)\geq p_k^{m_k -1} \phi\left(\frac{m}{p_k^{m_k}}\right)$$
for $2\leq k\leq t$ and $1\leq i\leq k-1$, where the inequality is strict except when $k=2$, $(p_1,p_2)=(2,3)$ and $m_1\geq 2$.
\end{lemma}

\begin{lemma}\label{inequ-1-1}
Let $m=p_1^{m_1}p_2^{m_2}\cdots p_t^{m_t}$, where $t\geq 2$, $m_1,m_2,\ldots, m_t$ are positive integers and $p_1,p_2,\ldots,p_t$ are prime numbers with $p_1<p_2<\cdots <p_t$. Then
$$\phi\left(\frac{m}{p_i}\right)\geq \phi\left(\frac{m}{p_k}\right)$$
for $1\leq i<k\leq t$, where equality holds if and only if $(k,p_1,p_2)=(2,2,3)$ with $m_1\geq 2$ and $m_2 =1$.
\end{lemma}

\begin{proof}
Since $k\geq 2$, we have $p_k\geq 3$. Then $p_k^{m_k -1} \phi\left(\frac{m}{p_k^{m_k}}\right)\geq \phi\left(\frac{m}{p_k}\right)$ with equality if and only if $m_k=1$. Now the result follows from Lemma \ref{inequ-2}.
\end{proof}

\noindent The following lemma follows by expanding $\phi(p_1p_2\cdots p_t)=(p_1 -1)(p_2 -1)\cdots (p_t -1)$.

\begin{lemma}\label{inequ-2-1}
Let $p_1,p_2, \ldots,p_t$ be pairwise distinct prime numbers with $t\geq 1$. Then
$$p_1p_2\cdots p_t-\phi(p_1p_2\cdots p_t)= \underset{i=1}{\overset{t}{\sum}}\frac{p_{1}\cdots p_{t}}{p_{i}} - \underset{i<j}{\underset{i,j=1}{\overset{t}{\sum}}}\frac{p_{1}\cdots p_{t}}{p_{i} p_{j}} + \underset{i<j<k}{\underset{i,j,k=1}{\overset{t}{\sum}}}\frac{p_{1}\cdots p_{t}}{p_ip_{j} p_{k}} +\ldots +(-1)^{t-1}.$$
\end{lemma}

We denote by $[m]$ the set $\{1,2,\dots,m\}$ for a given positive integer $m$. The following lemma can be seen using the fact that $\underset{d\mid m}{\sum} \phi(d) = m$ for every positive integer $m$.

\begin{lemma}\label{inequ-3}
Let $p_1,p_2, \ldots,p_t$ be prime numbers with $p_1<p_2<\cdots <p_t$. Then $p_{i_1}p_{i_2}\cdots p_{i_k}-\phi(p_{i_1}p_{i_2}\cdots p_{i_k}) \geq p_1p_2\cdots p_{k}-\phi(p_1p_2\cdots p_{k})$ for any subset $\{i_1,i_2,\ldots,i_k\}$ of $[t]$, with equality if and only if $k=1$ or $\{i_1,i_2,\ldots,i_k\}=[k]$.
\end{lemma}

We shall use the following fact throughout the paper. If $G_1,G_2,\ldots ,G_k$ are subgroups of the cyclic group $C_n$, then the number of elements in the intersection $G_1\cap G_2\cap\ldots\cap G_k$ is equal to the greatest common divisor of the integers $|G_1|,|G_2|,\ldots ,|G_k|$.

\begin{lemma}\label{none-1}
Let $n=p_1^{n_1}p_2^{n_2}\cdots p_r^{n_r}$, where $r\geq 2$, $n_1,n_2,\ldots, n_r$ are positive integers and $p_1,p_2,\ldots,p_r$ are pairwise distinct prime numbers.
Let $\left\{a_1,a_2,\ldots,a_s\right\}$ and $\left\{b_1,b_2,\ldots, b_t\right\}$ be two disjoint subsets of $[r]$, where $s\geq 1$, $t\geq 1$ and $s+t \leq r$. If $K$ is the union of the subgroups $S_{\frac{n}{p_{a_{1}}p_{b_1}\cdots p_{b_t}}}$, $S_{\frac{n}{p_{a_{2}}p_{b_1}\cdots p_{b_t}}}$, $\ldots$, $S_{\frac{n}{p_{a_{s}}p_{b_1}\cdots p_{b_t}}}$ of $C_n$, then
$$|K|=\frac{n}{p_1p_2\cdots p_r}\times \left[\frac{p_{1}p_2\cdots p_{r}}{p_{b_1}\cdots p_{b_t}} - p_{c_{1}}\cdots p_{c_{u}}\phi(p_{a_{1}}\cdots p_{a_{s}})\right],$$
where $s+t+u=r$ and $\{c_1,c_2,\ldots,c_u\}=[r]\setminus\{a_1,\ldots,a_s,b_1,\ldots, b_t\}$.
\end{lemma}

\begin{proof}
Since $K=\underset{j=1}{\overset{s}{\bigcup}} S_{\frac{n}{p_{a_{j}}p_{b_1}\cdots p_{b_t}}}$, we get that
\begin{align*}
\left\vert K\right\vert & = \underset{j=1}{\overset{s}{\sum}} \left\vert S_{\frac{n}{p_{a_{j}}p_{b_1}\cdots p_{b_t}}}\right\vert  - \underset{j<k}{\underset{j,k=1}{\overset{s}{\sum}}} \left\vert S_{\frac{n}{p_{a_{j}}p_{b_1}\cdots p_{b_t}}}\bigcap S_{\frac{n}{p_{a_{k}}p_{b_1}\cdots p_{b_t}}}\right\vert + \cdots + (-1)^{s-1} \left\vert \underset{j=1}{\overset{s}\bigcap} S_{\frac{n}{p_{a_{j}}p_{b_1}\cdots p_{b_t}}}\right\vert \\
& = \underset{j=1}{\overset{s}{\sum}} \frac{n}{p_{a_{j}}p_{b_1}\cdots p_{b_t}} - \underset{j<k}{\underset{j,k=1}{\overset{s}{\sum}}} \frac{n}{p_{a_{j}}p_{a_{k}}p_{b_1}\cdots p_{b_t}} + \cdots + (-1)^{s-1} \frac{n}{p_{a_{1}}\cdots p_{a_{s}}p_{b_1}\cdots p_{b_t}} \\
 & = \frac{n}{p_1p_2\cdots p_r}\times p_{c_{1}}\cdots p_{c_{u}}\times \left[\underset{j=1}{\overset{s}{\sum}}\frac{p_{a_{1}}\cdots p_{a_{s}}}{p_{a_{j}}} - \underset{j<k}{\underset{j,k=1}{\overset{s}{\sum}}}\frac{p_{a_{1}}\cdots p_{a_{s}}}{p_{a_{j}} p_{a_{k}}} +\cdots +(-1)^{s-1}\right]\\
 & = \frac{n}{p_1p_2\cdots p_r}\times p_{c_{1}}\cdots p_{c_{u}}\times \left[p_{a_{1}}\cdots p_{a_{s}} - \phi\left(p_{a_{1}}\cdots p_{a_{s}}\right) \right]\\
 & = \frac{n}{p_1p_2\cdots p_r}\times \left[\frac{p_{1}p_2\cdots p_{r}}{p_{b_1}\cdots p_{b_t}} - p_{c_{1}}\cdots p_{c_{u}}\phi(p_{a_{1}}\cdots p_{a_{s}})\right].
\end{align*}
We have used Lemma \ref{inequ-2-1} in the second last equality above.
\end{proof}

\section{Upper bounds} \label{revisit}

Let $n=p_1^{n_1}p_2^{n_2}\cdots p_r^{n_r}$, where $r\geq 2$, $n_1,n_2,\ldots, n_r$ are positive integers and $p_1,p_2,\ldots,p_r$ are prime numbers with $p_1<p_2<\cdots <p_r$.
For $1\leq j \leq r$, let $Y_j$ and $Z_j$ be the subsets of $C_n$ defined by
$$
Y_j :=E_n\bigcup \left(\underset{t\neq j}{\underset{t=1}{\overset{r}{\bigcup}}}S_{\frac{n}{p_jp_t}}\right),\hskip .5cm  Z_j :=E_n\bigcup\left({\underset{s=1}{\overset{n_j-1}{\bigcup}}}E_{\frac{n}{p_j^s}}\right)\bigcup \left(\underset{t\neq j}{\underset{t=1}{\overset{r}{\bigcup}}}S_{\frac{n}{p_j^{n_j}p_t}}\right).$$
Observe that $Z_j=Y_j$ if $n_j=1$. The following argument showing that $Y_j$ and $Z_j$ are cut-sets of $\mathcal{P}(C_n)$ is similar to the proof of \cite[Proposition 3.1]{cps}.\\

\noindent {\bf $Y_j$ is a cut-set of $\mathcal{P}(C_n)$:} For an element $x\in \overline{Y_j}$, observe that $o(x)$ is one of the following two types:
\begin{enumerate}
\item[(Y1)] $p_1^{n_1}\cdots p_{j-1}^{n_{j-1}} p_{j}^{s} p_{j+1}^{n_{j+1}}\cdots p_{r}^{n_r}$ for some $s\in\{0,1,\ldots, n_j -1\}$,
\item[(Y2)] $p_1^{l_1}\cdots p_{j-1}^{l_{j-1}} p_{j}^{n_j} p_{j+1}^{l_{j+1}}\cdots p_{r}^{l_r}$, where $0\leq l_i\leq n_i$ for each $i\in [r]\setminus\{j\}$ and $l_i\neq n_i$ for at least one $i\in [r]\setminus\{j\}$.
\end{enumerate}
Let $A_j$ (respectively, $B_j$) be the subset of $\overline{Y_j}$ consisting of all the elements whose order is of type (Y1) (respectively, type (Y2)). Then $A_j,B_j$ are nonempty sets and $A_j\cup B_j=\overline{Y_j}$.
Since $l_i\neq n_i$ for at least one $i\in [r]\setminus\{j\}$, no element of $A_j$ can be obtained as an integral power of any element of $B_j$. Since $s < n_j$, no element of $B_j$ can be obtained as a power of any element of $A_j$. Thus there is no edge of $\mathcal{P}(\overline{Y_j})$ with one vertex from $A_j$ and the other one from $B_j$. Therefore, $A_j\cup B_j$ is a separation of $\mathcal{P}(\overline{Y_j})$ and hence $Y_j$ is a cut-set of $\mathcal{P}(C_n)$.\\

\noindent {\bf $Z_j$ is a cut-set of $\mathcal{P}(C_n)$:} For an element $x\in \overline{Z_j}$, observe that $o(x)$ is one of the following two types:
\begin{enumerate}
\item[(Z1)] $p_1^{n_1}\cdots p_{j-1}^{n_{j-1}}p_{j+1}^{n_{j+1}}\cdots p_r^{n_r}$,
\item[(Z2)] $p_1^{l_1}\cdots p_{j-1}^{l_{j-1}}p_j^{t}p_{j+1}^{l_{j+1}}\cdots p_r^{l_r}$, where $1\leq t\leq n_j$, $0 \leq l_i\leq n_i$ for $i\in [r]\setminus\{j\}$ and  $l_i\neq n_i$ for at least one $i\in [r]\setminus\{j\}$.
\end{enumerate}
Let $K_j$ (respectively, $L_j$) be the subset of $\overline{Z_j}$ consisting of all the elements whose order is of type (Z1) (respectively, type (Z2)). A similar argument as in the case of $Y_j$ implies that $K_j\cup L_j$ is a separation of $\mathcal{P}(\overline{Z_j})$ and so $Z_j$ is a cut-set of $\mathcal{P}(C_n)$.\\

\noindent Now, for $1\leq j\leq r$, consider the integers $\alpha_j(n)$ and $\beta_j(n)$ defined by
\begin{align*}
\alpha_j(n)& :=\phi(n) + \frac{n}{p_1p_2 \cdots p_r}\times \left[\frac{p_1p_2\cdots p_{r}}{p_j}-\phi\left(\frac{p_1p_2\cdots p_{r}}{p_j}\right)\right],\\
\beta_j(n) & :=\phi(n) + \frac{n}{p_1p_2\cdots p_r}\times \frac{1}{p_j^{n_j-1}}\left[\frac{p_1p_2\cdots p_{r}}{p_j} +\phi\left(\frac{p_1p_2\cdots p_{r}}{p_j}\right)\left( p_j^{n_j-1}-2\right) \right].
\end{align*}
Note that $\alpha_r(n)=\alpha(n)$ and $\beta_r(n)=\beta(n)$, where $\alpha(n)$ and $\beta(n)$ are defined in the first section.

\begin{lemma}
$|Y_j|=\alpha_j(n)$ and $|Z_j|=\beta_j(n)$.
\end{lemma}

\begin{proof}
Using Lemma \ref{none-1},
$$|Y_j|= |E_n| + \left\vert \underset{t\neq j}{\underset{t=1}{\overset{r}{\bigcup}}}S_{\frac{n}{p_jp_t}} \right\vert = \phi(n)+ \frac{n}{p_1\cdots p_r}\times \left[\frac{p_1\cdots p_r}{p_j}-\phi\left(\frac{p_1\cdots p_r}{p_j} \right) \right]=\alpha_j(n).$$
Since the sets $E_{\frac{n}{p_j}}, E_{\frac{n}{p_j^{2}}},\ldots, E_{\frac{n}{p_j^{n_j -1}}}$ are pairwise disjoint, we have
\begin{align*}
\left\vert{\underset{s=1}{\overset{n_j-1}{\bigcup}}}E_{\frac{n}{p_j^s}}\right\vert = {\underset{s=1}{\overset{n_j-1}{\sum}}} \left\vert E_{\frac{n}{p_j^s}}\right\vert & =\frac{n}{p_1\cdots p_r}\times \frac{1}{p_j^{n_j -1}}\times\phi\left(\frac{p_1\cdots p_r}{p_j} \right)\times \left[\phi\left(p_j^{n_j -1} \right)+ \cdots + \phi\left(p_j\right)\right]\\
 & =\frac{n}{p_1\cdots p_r}\times \frac{1}{p_j^{n_j -1}}\times\phi\left(\frac{p_1\cdots p_r}{p_j} \right)\left[p_j^{n_j -1} -1 \right].
\end{align*}
Again using Lemma \ref{none-1}, it can be calculated that
$$\left\vert \underset{t\neq j}{\underset{t=1}{\overset{r}{\bigcup}}}S_{\frac{n}{p_j^{n_j}p_t}} \right\vert = \frac{n}{p_1\cdots p_r}\times \frac{1}{p_j^{n_j -1}}\times\left[\frac{p_1\cdots p_r}{p_j}-\phi\left(\frac{p_1\cdots p_r}{p_j} \right) \right].$$
Therefore,
\begin{align*}
|Z_j| & = |E_n| + \left\vert{\underset{s=1}{\overset{n_j-1}{\bigcup}}}E_{\frac{n}{p_j^s}}\right\vert +\left\vert \underset{t\neq j}{\underset{t=1}{\overset{r}{\bigcup}}}S_{\frac{n}{p_j^{n_j}p_t}}\right\vert\\
 & = \phi(n)+ \frac{n}{p_1\cdots p_r}\times \frac{1}{p_j^{n_j -1}} \times \left[\frac{p_1\cdots p_r}{p_j}+\phi\left(\frac{p_1\cdots p_r}{p_j} \right)\left(p_j^{n_j -1} -2 \right) \right]=\beta_j(n).
 \end{align*}
 This completes the proof.
\end{proof}

\begin{lemma}\label{compa-1}
If $r\geq 3$, then $\alpha_1(n)>\alpha_2(n)>\cdots >\alpha_r(n)$.
\end{lemma}

\begin{proof}
Let $j,k\in [r]$ with $j<k$. Since $p_j < p_k$ and $r\geq 3$, we have
\begin{align*}
\alpha_j(n) -\alpha_k(n)&= \frac{n}{p_1\cdots p_r} \times \left((p_k -p_j)\left[ \frac{p_1\cdots p_r}{p_j p_k} - \phi\left(\frac{p_1\cdots p_r}{p_j p_k}\right)\right] \right)>0
\end{align*}
and so the lemma follows.
\end{proof}

\begin{lemma}\label{compa-2}
If $r\geq 3$, then the following hold:
\begin{enumerate}
\item[(i)] $\alpha_j(n) = \beta_j(n)$ if and only if $n_j=1$.
\item[(ii)] $\alpha_j(n) < \beta_j(n)$ if and only if $n_j \geq 2$ and $2\phi\left(\frac{p_1p_2\cdots p_{r}}{p_j}\right) > \frac{p_1p_2\cdots p_{r}}{p_j}$.
\item[(iii)] $\alpha_j(n) > \beta_j(n)$ if and only if $n_j \geq 2$ and $2\phi\left(\frac{p_1p_2\cdots p_{r}}{p_j}\right) < \frac{p_1p_2\cdots p_{r}}{p_j}$.
\end{enumerate}
\end{lemma}

\begin{proof}
We have
\begin{align*}
\alpha_j(n) -\beta_j(n)&= \frac{n}{p_1\cdots p_r} \times \left(1-\frac{1}{p_j^{n_j -1}}\right)\left[\frac{p_1\cdots p_r}{p_j} - 2\phi\left(\frac{p_1\cdots p_r}{p_j}\right)\right].
\end{align*}
Since $r\geq 3$, $2\phi\left(\frac{p_1p_2\cdots p_{r}}{p_j}\right) \neq \frac{p_1p_2\cdots p_{r}}{p_j}$ for $j\in [r]$, see \cite[Theorem 1.3(iii)]{cps}. It now follows that (i), (ii) and (iii) hold.
\end{proof}

\begin{lemma}\label{compa-3}
Let $r\geq 3$ and $j,k\in [r]$ with $j< k$. If $2\phi\left(\frac{p_1p_2\cdots p_{r}}{p_j}\right) < \frac{p_1p_2\cdots p_{r}}{p_j}$, then $\beta_j(n)>\beta_k(n)$.
\end{lemma}

\begin{proof}
Since $2\phi\left(\frac{p_1p_2\cdots p_{r}}{p_j}\right) < \frac{p_1p_2\cdots p_{r}}{p_j}$ and $\frac{\phi(p_j)}{\phi(p_k)} < \frac{p_j}{p_k}$, we have
$$2\phi\left(\frac{p_1\cdots p_{r}}{p_k}\right) =2\phi\left(\frac{p_1\cdots p_{r}}{p_j}\right)\times \frac{\phi(p_j)}{\phi(p_k)} < \frac{p_1\cdots p_{r}}{p_j}\times \frac{p_j}{p_k}=\frac{p_1\cdots p_{r}}{p_k}.$$
Write $\beta_i(n)=\phi(n)+\frac{n}{p_1\cdots p_r}\times u_i$ for $i\in\{j,k\}$, where
$$u_i=\frac{1}{p_i^{n_i -1}}\times \left[\frac{p_1\cdots p_r}{p_i} - 2\phi\left(\frac{p_1\cdots p_r}{p_i}\right)\right] + \phi\left(\frac{p_1\cdots p_r}{p_i}\right).$$
If $p_j^{n_j -1}\leq p_k^{n_k -1}$, then it can be calculated that
$$u_j -u_k \geq \frac{p_k -p_j}{p_j^{n_j -1}}\times \left[\frac{p_1\cdots p_r}{p_j p_k} - \phi\left(\frac{p_1\cdots p_r}{p_j p_k}\right) +\left(p_j^{n_j-1}-1\right)\phi\left(\frac{p_1\cdots p_r}{p_j p_k}\right) \right] > 0.$$
The last strict inequality holds as $r\geq 3$. Similarly, if $p_j^{n_j -1}\geq  p_k^{n_k -1}$, then
$$u_j -u_k \geq \frac{p_k -p_j}{p_k^{n_k -1}}\times \left[\frac{p_1\cdots p_r}{p_j p_k} - \phi\left(\frac{p_1\cdots p_r}{p_j p_k}\right) +\left(p_k^{n_k-1}-1\right)\phi\left(\frac{p_1\cdots p_r}{p_j p_k}\right) \right] > 0.$$
In both cases, it follows that $\beta_j(n)>\beta_k(n)$.
\end{proof}

\subsection{Proof of Theorem \ref{res-1}}\label{new-bound}

Assume that $r\geq 3$ and fix $a,b\in [r]$. Let $R$ be the union of the $r-2$ subgroups $S_{\frac{n}{p_ip_{a}p_b}}$ of $C_n$, where $i\in [r]\setminus\{a,b\}$. Now consider the set $X_{a,b}$ defined by
$$X_{a,b}:=R\bigcup E_n\bigcup E_{\frac{n}{p_a}}\bigcup E_{\frac{n}{p_a^{2}}}\bigcup \ldots \bigcup E_{\frac{n}{p_a^{n_a}}}\bigcup E_{\frac{n}{p_{b}}}\bigcup E_{\frac{n}{p_{b}^{2}}}\bigcup \ldots \bigcup E_{\frac{n}{p_{b}^{n_{b}}}}.$$
Observe that the sets involved in the definition of $X_{a,b}$ are pairwise disjoint. Define the integer $\gamma_{a,b}(n)$ by
$$\gamma_{a,b}(n):=\phi(n)+\frac{n}{p_1\cdots p_r}\times \left[\phi\left(\frac{p_1\cdots p_{r}}{p_a}\right)+ \phi\left(\frac{p_1\cdots p_{r}}{p_b}\right)+\frac{p_1\cdots p_{r}}{p_a p_b}-\phi\left(\frac{p_1\cdots p_{r}}{p_ap_b}\right)\right].$$
We next show that $X_{a,b}$ is a cut-set of $\mathcal{P}(C_n)$ and that $|X_{a,b}|=\gamma_{a,b}(n)$.

\begin{proposition}\label{new-cut}
$X_{a,b}$ is a cut-set of $\mathcal{P}(C_n)$.
\end{proposition}

\begin{proof}
We prove the proposition by producing a separation of $\mathcal{P}(\overline{X_{a,b}})$. Without loss, we may assume that $a<b$. For an arbitrary element $x\in \overline{X_{a,b}}$, observe that $o(x)$ is of the form:
$$o(x)=p_1^{l_1}\cdots p_{a-1}^{l_{a-1}}p_a^{l_a} p_{a+1}^{l_{a+1}}\cdots p_{b-1}^{l_{b-1}}p_b^{l_b} p_{b+1}^{l_{b+1}}\cdots p_r^{l_r},$$
where the integers $l_i$ satisfy the following conditions:
\begin{enumerate}
\item[(i)] $0\leq l_i\leq n_i$ for each $i\in [r]$.
\item[(ii)] If $l_{a}<n_{a}$ and $l_b< n_b$, then $l_i=n_i$ for each $i\in [r]\setminus\{a,b\}$.

[Otherwise, $x$ would be in $S_{\frac{n}{p_ip_{a}p_b}}$ for some $i\in [r]\setminus\{a,b\}$ and so in $R$.]
\item[(iii)] If $l_{a}=n_{a}$ or $l_b=n_b$, then $l_i\neq n_i$ for at least one $i\in [r]\setminus\{a,b\}$.

[Otherwise, $x$ would be in $E_n$ if $l_{a}=n_{a}$ and $l_b=n_b$, or in $E_{\frac{n}{p_{a}^{j}}}$ for some $j\in [n_{a}]$ if $l_{a}< n_{a}$ and $l_b=n_b$, or in $E_{\frac{n}{p_b^{k}}}$ for some $k\in [n_b]$ if $l_{a}=n_{a}$ and $l_b < n_b$.]
\end{enumerate}
Let $A$ be the subset of $\overline{X_{a,b}}$ consisting of all the elements whose order satisfy $l_{a}<n_{a}$ and $l_b< n_b$. Take $B:=\overline{X_{a,b}}\setminus A$. Then each of $A,B$ is nonempty as $r\geq 3$, and $\overline{X_{a,b}} = A\cup B$ is a disjoint union.
Since $l_{a}=n_{a}$ or $l_b=n_b$ for the order of each element of $B$, no element of $B$ can be obtained as an integral power of any element of $A$.
Again, since $l_i\neq n_i$ for at least one $i\in [r]\setminus\{a,b\}$ for the order of each element of $B$, no element of $A$ can be obtained as an integral power of any element of $B$. Thus there is no edge of $\mathcal{P}(\overline{X_{a,b}})$ with one vertex from $A$ and the other one from $B$. Therefore, $A\cup B$ is a separation of $\mathcal{P}(\overline{X_{a,b}})$.
\end{proof}

\begin{proposition}\label{new-size}
$|X_{a,b}|=\gamma_{a,b}(n)$.
\end{proposition}

\begin{proof}
Applying Lemma \ref{none-1}, we get
$$|R|= \frac{n}{p_1p_2\cdots p_r}\times \left[\frac{p_1p_2\cdots p_{r}}{p_a p_b}-\phi\left(\frac{p_1p_2\cdots p_{r}}{p_a p_b}\right)\right].$$
Since the sets $E_{\frac{n}{p_a}}, E_{\frac{n}{p_a^{2}}},\ldots, E_{\frac{n}{p_a^{n_a}}}$ are pairwise disjoint, we have
\begin{align*}
\left\vert\underset{j=1}{\overset{n_a}{\bigcup}} E_{\frac{n}{p_a^j}}\right\vert = \underset{j=1}{\overset{n_a}{\sum}} \left\vert E_{\frac{n}{p_a^j}}\right\vert  = \underset{j=1}{\overset{n_a}{\sum}} \phi\left(\frac{n}{p_a^j}\right) & =\phi\left(\frac{n}{p_a^{n_a}}\right)\times \underset{j=0}{\overset{n_a -1}{\sum}} \phi\left(p_a^j\right)\\
& =\phi\left(\frac{n}{p_a^{n_a}}\right) p_a^{n_a -1}  =  \frac{n}{p_1\cdots p_r} \times \phi\left(\frac{p_1\cdots p_r}{p_{a}}\right).
\end{align*}
A similar calculation gives that
$$\left\vert\underset{j=1}{\overset{n_{b}}{\bigcup}} E_{\frac{n}{p_{b}^j}}\right\vert =  \frac{n}{p_1p_2\cdots p_r} \times \phi\left(\frac{p_1p_2\cdots p_r}{p_{b}}\right).$$
Since the right hand side in the definition of the set $X_{a,b}$ is a disjoint union, we have
$$|X_{a,b}|=|E_n|+|R|+\left\vert\underset{j=1}{\overset{n_a}{\bigcup}} E_{\frac{n}{p_a^j}}\right\vert+\left\vert\underset{j=1}{\overset{n_{b}}{\bigcup}} E_{\frac{n}{p_{b}^j}}\right\vert.$$
Using $|E_n|=\phi(n)$, it now follows from the above equalities that $|X_{a,b}|=\gamma_{a,b}(n)$.
\end{proof}

\noindent As a consequence of Propositions \ref{new-cut} and \ref{new-size}, we have

\begin{corollary}
$\kappa(\mathcal{P}(C_n))\leq \gamma_{a,b}(n)$ for any two distinct $a,b\in [r]$.
\end{corollary}

\begin{proof}[{\bf Proof of Theorem \ref{res-1}}]
The integer $\gamma(n)$ defined in the first section is precisely the integer $\gamma_{r-1,r}(n)$. So the theorem follows from the above corollary.
\end{proof}

\begin{proposition}
$\gamma_{a,b}(n) > \gamma_{r-1,r}(n)=\gamma(n)$ for $a,b\in [r]$ with $\{a,b\}\neq \{r-1, r\}$.
\end{proposition}

\begin{proof}
Since $\{a,b\}\neq \{r-1, r\}$, we have
$$\phi\left(\frac{p_1p_2\cdots p_r}{p_r}\right)+ \phi\left(\frac{p_1p_2\cdots p_r}{p_{r-1}}\right)<\phi\left(\frac{p_1p_2\cdots p_r}{p_a}\right)+ \phi\left(\frac{p_1p_2\cdots p_r}{p_{b}}\right).$$
By Lemma \ref{inequ-3},
$$\left[\frac{p_1p_2\cdots p_{r}}{p_{r-1} p_r}-\phi\left(\frac{p_1p_2\cdots p_{r}}{p_{r-1} p_r}\right)\right]\leq \left[\frac{p_1p_2\cdots p_{r}}{p_a p_b}-\phi\left(\frac{p_1p_2\cdots p_{r}}{p_a p_b}\right)\right].$$
Now it can be seen that the proposition holds.
\end{proof}

\begin{proposition}\label{compa-4}
The following hold:
\begin{enumerate}
\item[(i)] If $n_r\geq 2$, then $\beta(n) < \gamma(n)$.
\item[(ii)] $\alpha(n) \leq \gamma(n)$ if and only if $\left(2+\frac{p_r -2}{p_{r-1}-1}\right)\times \phi \left( p_1p_2\cdots p_{r-2}\right)\geq p_1p_2\cdots p_{r-2}$.
\end{enumerate}
\end{proposition}

\begin{proof}
We have $\beta(n)=\phi(n)+ \frac{n}{p_1\cdots p_r}\times u$ and $\gamma(n)=\phi(n)+ \frac{n}{p_1\cdots p_r}\times v$, where
\begin{align*}
u & = \frac{1}{p_r^{n_r-1}}\left[p_1p_2\cdots p_{r-1} +\phi(p_1\cdots p_{r-1})\left( p_r^{n_r-1}-2\right) \right],\\
v & = \phi(p_1\cdots p_{r-1})+ \phi(p_1\cdots p_{r-2}p_r)+p_1\cdots p_{r-2}-\phi(p_1\cdots p_{r-2}).
\end{align*}
An easy calculation gives that
\begin{align*}
v - u & = \frac{1}{p_r^{n_r-1}}\left[\phi(p_1\cdots p_{r-2}) \left( p_r^{n_r -1}(p_r -2) + 2\phi(p_{r-1})\right) + p_1\cdots p_{r-2}\left(p_r^{n_r-1} - p_{r-1} \right)\right].
\end{align*}
Now $n_r\geq 2$ implies that $v-u > 0$ and it follows that $\gamma(n) >\beta(n)$. This proves (i).\\
It can be calculated that
$$\gamma(n)-\alpha(n) =\frac{n}{p_1\cdots p_r}\times \phi(p_{r-1})\left[\left(2+\frac{p_r -2}{p_{r-1}-1}\right)\times \phi \left( p_1p_2\cdots p_{r-2}\right)- p_1p_2\cdots p_{r-2}\right]$$
and (ii) follows from this.
\end{proof}

\subsection{Application of Theorem \ref{res-1}}

As an application of Theorem \ref{res-1}, we prove the following result which is useful while proving Theorems \ref{res-2} and \ref{res-3} in the subsequent sections.

\begin{proposition}\label{two-atmost}
If $X$ is a minimum cut-set of $\mathcal{P}(C_n)$, then $X$ contains at most two of the sets $E_{\frac{n}{p_1}}, E_{\frac{n}{p_2}},\dots, E_{\frac{n}{p_r}}$.
\end{proposition}

\begin{proof}
By Theorem \ref{res-1}, we have $|X|=\kappa(\mathcal{P}(C_n))\leq \gamma(n)$.
Set $T:=X\setminus E_n$. Since $E_n$ is contained in $X$, we have $|T|=|X|-|E_n|=|X|-\phi(n)\leq \gamma(n)-\phi(n)$. This gives
\begin{equation}\label{eqn-2}
|T| \leq \frac{n}{p_1\cdots p_r}\times \left[\phi(p_1\cdots p_{r-1})+ \phi(p_1\cdots p_{r-2}p_r)+p_1\cdots p_{r-2}-\phi(p_1\cdots p_{r-2})\right].
\end{equation}
Consider the sets $E_{\frac{n}{p_i}}$ for $1\leq i\leq r$. By Lemma \ref{simple}, each of them is either contained in $X$ or disjoint from $X$. Note that, if any such set is contained in $X$, then it must be contained in $T$.

If possible, suppose that the sets $E_{\frac{n}{p_j}}, E_{\frac{n}{p_k}}$ and $E_{\frac{n}{p_l}}$ are contained in $X$, where $j,k,l$ are pairwise distinct elements in $[r]$. Without loss, we may assume that $j<k<l$.
Since $E_{\frac{n}{p_j}}, E_{\frac{n}{p_k}}, E_{\frac{n}{p_l}}$ are pairwise disjoint and contained in $T$, we have
$$|T|\geq \left\vert E_{\frac{n}{p_j}}\right\vert + \left\vert E_{\frac{n}{p_k}}\right\vert +\left\vert E_{\frac{n}{p_l}}\right\vert =\phi\left(\frac{n}{p_j}\right)+\phi\left(\frac{n}{p_k}\right)+\phi\left(\frac{n}{p_l}\right).$$
Since $j\leq r-2$ and $k\leq r-1$, Lemma \ref{inequ-2} gives that
$$\phi\left(\frac{n}{p_j}\right)\geq p_{r-1}^{n_{r-1}-1} \phi\left(\frac{n}{p_{r-1}^{n_{r-1}}}\right)=\frac{n}{p_1\cdots p_r}\times \phi\left(p_1\cdots p_{r-2}p_r \right)$$
and
$$\phi\left(\frac{n}{p_k}\right)\geq p_r^{n_r -1} \phi\left(\frac{n}{p_r^{n_r}}\right)=\frac{n}{p_1\cdots p_r}\times \phi\left(p_1\cdots p_{r-2}p_{r-1} \right).$$
If $l<r$, then again using Lemma \ref{inequ-2}, we have
\begin{align*}
\phi\left(\frac{n}{p_l}\right)\geq p_{r}^{n_{r}-1} \phi\left(\frac{n}{p_{r}^{n_{r}}}\right) & =\frac{n}{p_1\cdots p_r}\times \phi\left(p_1\cdots p_{r-2}p_{r-1} \right)\\
 & = \frac{n}{p_1\cdots p_r}\times \phi\left(p_1\cdots p_{r-2}\right) (p_{r-1} -1)\\
 & > \frac{n}{p_1\cdots p_r}\times \left[p_1p_2\cdots p_{r-2} - \phi\left(p_1\cdots p_{r-2}\right)\right].
\end{align*}
The strict inequality in the above holds by Lemma \ref{inequ-1} as $p_{r-1}> r-1$.
Now suppose that $l=r$. We have $\frac{\phi(p_{r-1}p_r)}{p_r} = \phi(p_{r-1}) - \frac{\phi(p_{r-1})}{p_r}> r-2$.
Then using Lemmas \ref{inequ-1} and \ref{inequ-2} again, we have
\begin{align*}
\phi\left(\frac{n}{p_l}\right) = \phi\left(\frac{n}{p_r}\right) & \geq \frac{n}{p_1\cdots p_r}\times \frac{\phi\left(p_1p_2\cdots p_{r-1}p_r\right)}{p_r}\\
 & > \frac{n}{p_1\cdots p_r}\times (r-2)\phi\left(p_1\cdots p_{r-2}\right)\\
 & = \frac{n}{p_1\cdots p_r}\times \left[(r-1)\phi\left(p_1\cdots p_{r-2}\right) - \phi\left(p_1\cdots p_{r-2}\right)\right]\\
 & \geq \frac{n}{p_1\cdots p_r}\times \left[p_1p_2\cdots p_{r-2} - \phi\left(p_1\cdots p_{r-2}\right)\right].
\end{align*}
Combining the above inequalities, we get
$$|T| > \frac{n}{p_1\cdots p_r}\times \left[\phi(p_1\cdots p_{r-1})+ \phi(p_1\cdots p_{r-2}p_r)+p_1\cdots p_{r-2}-\phi(p_1\cdots p_{r-2})\right],$$
which is a contradiction to $(\ref{eqn-2})$. This completes the proof.
\end{proof}

\section{Proof of Theorem \ref{res-2}}

Let $n=p_1^{n_1}p_2^{n_2}\cdots p_r^{n_r}$, where $r\geq 3$, $n_1,n_2,\ldots, n_r$ are positive integers and $p_1,p_2,\ldots,p_r$ are prime numbers with $p_1<p_2<\cdots <p_r$.
Recall that the set
$$Y_r =E_n\bigcup \left(\underset{t=1}{\overset{r-1}{\bigcup}}S_{\frac{n}{p_rp_t}}\right)$$
defined in Section \ref{revisit} is a cut-set of $\mathcal{P}(C_n)$ which gives the upper bound $\alpha_r(n)=\alpha(n)$ for $\kappa(\mathcal{P}(C_n))$.

We shall use the following fact frequently in the rest of the paper. Suppose that $X$ is a cut-set of $\mathcal{P}(C_n)$ and $A\cup B$ is a separation of $\mathcal{P}(\overline{X})$. If $A$ contains an element of order $a$ and $B$ contains an element of order $b$, then the unique subgroup $S_{(a,b)}$ of $C_n$ must be contained in $X$, where $(a,b)$ denotes the greatest common divisor of $a$ and $b$. The following proposition gives a sufficient condition for $\kappa(\mathcal{P}(C_n))= \alpha(n)$.

\begin{proposition}\label{all}
Let $X$ be a minimum cut-set of $\mathcal{P}(C_n)$. If $X$ does not contain any of the sets $E_{\frac{n}{p_1}}, E_{\frac{n}{p_2}},\dots, E_{\frac{n}{p_r}}$, then $X=Y_r$ and so
$\kappa(\mathcal{P}(C_n))=\alpha(n)$.
\end{proposition}

\begin{proof}
Fix a separation $A\cup B$ of $\mathcal{P}(\overline{X})$.
Let $P=\left\{E_{\frac{n}{p_i}}: 1\leq i\leq r, E_{\frac{n}{p_i}}\subseteq A\right\}$ and $Q=\left\{E_{\frac{n}{p_j}}: 1\leq j\leq r, E_{\frac{n}{p_j}}\subseteq B\right\}$.
Set $a=|P|$ and $b=|Q|$. Then $a+b=r$ with $1\leq a\leq r-1$ and $1\leq b\leq r-1$ (these inequalities hold since both $A$ and $B$ are nonempty and $E_n$ is contained in $X$).\\

\noindent {\bf Case 1:} $a=1$ or $b=1$.\medskip

Without loss, we may assume that $a=1$. Suppose that $P=\left\{E_{\frac{n}{p_k}}\right\}$ for some $k\in [r]$. Then $Q=\left\{E_{\frac{n}{p_j}}:j\in [r]\setminus\{k\}\right\}$. We show that $k=r$.

Since $A\cup B$ is a separation of $\mathcal{P}(\overline{X})$, it follows that the subgroups
$S_{\frac{n}{p_{k}p_{j}}},\; j\in [r]\setminus\{k\}$, of $C_n$ are contained in $X$. Let $L$ be the union of these $r-1$ subgroups of $C_n$. By Lemma \ref{none-1},
$$|L|= \frac{n}{p_1p_2 \cdots p_r}\times \left[\frac{p_1p_2\cdots p_r}{p_k}-\phi\left(\frac{p_1p_2\cdots p_r}{p_k}\right)\right].$$
Since $E_n$ and $L$ are disjoint and contained in $X$, we get $|X|\geq \phi(n) + |L|$. If $k\neq r$, then Lemma \ref{inequ-3} together with the fact that $r\geq 3$ imply
$$\kappa(\mathcal{P}(C_n))=|X| \geq \phi(n) + |L| > \phi(n) + \frac{n}{p_1p_2 \cdots p_r}\times \left[p_1p_2\cdots p_{r-1}-\phi(p_1p_2\cdots p_{r-1})\right]=\alpha(n),$$
a contradiction to (\ref{bound-1}). 

Thus $k=r$ and hence $X$ contains $E_n\cup L=E_n\bigcup \left(\underset{j=1}{\overset{r-1}{\bigcup}}S_{\frac{n}{p_rp_j}}\right)=Y_r$. Since both $X$ and $Y_r$ are cut-sets of $\mathcal{P}(C_n)$ with $X$ being of minimum size, we must have $X=Y_r$ and hence $\kappa(\mathcal{P}(C_n))=|X|=|Y_r|=\alpha(n)$. \\

\noindent {\bf Case 2:} $a\geq 2$ and $b\geq 2$.\medskip

We show that this case is not possible. Suppose that $P=\left\{E_{\frac{n}{p_{i_1}}},E_{\frac{n}{p_{i_2}}},\ldots,E_{\frac{n}{p_{i_a}}}\right\}$ and $Q=\left\{E_{\frac{n}{p_{i_{a+1}}}},E_{\frac{n}{p_{i_{a+2}}}},\ldots,E_{\frac{n}{p_{i_{a+b}}}}\right\}$, where
$[r]=\{i_1,i_2,\ldots,i_a,i_{a+1},i_{a+2},\ldots, i_{a+b}\}$.
Without loss, we may assume that $p_{i_1} > p_{i_{a+1}}$. Since $A\cup B$ is a separation of $\mathcal{P}(\overline{X})$, the following subgroups
\begin{align*}
& S_{\frac{n}{p_{i_1}p_{i_{a+1}}}},\; S_{\frac{n}{p_{i_1}p_{i_{a+2}}}},\ldots,\; S_{\frac{n}{p_{i_1}p_{i_{a+b}}}}\\
& S_{\frac{n}{p_{i_2}p_{i_{a+1}}}}, \; S_{\frac{n}{p_{i_2}p_{i_{a+2}}}},\ldots,\; S_{\frac{n}{p_{i_2}p_{i_{a+b}}}}\\
& \vdots \\
& S_{\frac{n}{p_{i_a}p_{i_{a+1}}}},\; S_{\frac{n}{p_{i_a}p_{i_{a+2}}}},\ldots,\; S_{\frac{n}{p_{i_a}p_{i_{a+b}}}}
\end{align*}
of $C_n$ must be contained in $X$. Consider the following three subsets of $C_n$:
$$K_1=\underset{k=1}{\overset{b}{\bigcup}} S_{\frac{n}{p_{i_1}p_{i_{a+k}}}},\;\; K_2=\underset{l=2}{\overset{a}{\bigcup}} S_{\frac{n}{p_{i_l}p_{i_{a+1}}}},\;\; K_3=\underset{l=2}{\overset{a}{\bigcup}} S_{\frac{n}{p_{i_{1}}p_{i_l}}}.$$
Thus $K_1$ is the union of the subgroups listed above in the first row, and $K_2$ is the union of the subgroups listed in the first column, except the subgroup $S_{\frac{n}{p_{i_1}p_{i_{a+1}}}}$. The set $K_3$ is well-defined as $a\geq 2$.
Note that $K_1$ and $K_2$ are contained in $X$ but $K_3$ need not be. By Lemma \ref{none-1},
\begin{align*}
|K_2| & =\frac{n}{p_1p_2 \cdots p_r}\times p_{i_1} p_{i_{a+2}}\cdots p_{i_{a+b}}\left[p_{i_{2}}\cdots p_{i_{a}}-\phi(p_{i_{2}}\cdots p_{i_{a}})\right],\\
|K_3| & =\frac{n}{p_1p_2 \cdots p_r}\times p_{i_{a+1}} p_{i_{a+2}}\cdots p_{i_{a+b}}\left[p_{i_{2}}\cdots p_{i_{a}}-\phi(p_{i_{2}}\cdots p_{i_{a}})\right].
\end{align*}
Since $p_{i_1} > p_{i_{a+1}}$ by our assumption, it follows that $|K_2|>|K_3|$. We have
\begin{align*}
K_1\cap K_2 & =\left(\underset{l=2}{\overset{a}{\bigcup}} S_{\frac{n}{p_{i_1}p_{i_l}p_{i_{a+1}}}}\right)\bigcup \left(\underset{l=2}{\overset{a}{\bigcup}} S_{\frac{n}{p_{i_1}p_{i_l}p_{i_{a+1}}p_{i_{a+2}}}}\right)\bigcup \ldots \bigcup \left(\underset{l=2}{\overset{a}{\bigcup}} S_{\frac{n}{p_{i_1}p_{i_l}p_{i_{a+1}}p_{i_{a+b}}}}\right)\\
 & \subseteq \left(\underset{l=2}{\overset{a}{\bigcup}} S_{\frac{n}{p_{i_1}p_{i_l}p_{i_{a+1}}}}\right)\bigcup \left(\underset{l=2}{\overset{a}{\bigcup}} S_{\frac{n}{p_{i_1}p_{i_l}p_{i_{a+2}}}}\right)\bigcup \ldots \bigcup \left(\underset{l=2}{\overset{a}{\bigcup}} S_{\frac{n}{p_{i_1}p_{i_l}p_{i_{a+b}}}}\right)\\
& = K_1\cap K_3.
\end{align*}
So $|K_1\cap K_2|\leq |K_1\cap K_3|$ and hence
$$|K_1\cup K_2|=|K_1|+|K_2|-|K_1\cap K_2|> |K_1|+|K_3|-|K_1\cap K_3|= |K_1\cup K_3|.$$
Therefore, $|X|\geq |E_n|+|K_1\cup K_2| > \phi(n)+|K_1\cup K_3|$. Since $K_1\cup K_3$ is the union of the subgroups $S_{\frac{n}{p_{i_1}p_j}}$, $j\in [r]\setminus\{i_1\}$, of $C_n$,
Lemmas \ref{none-1} and Lemma \ref{inequ-3} give that
\begin{align*}
|K_1\cup K_3| & =\frac{n}{p_1p_2 \cdots p_r}\times \left[p_{i_2}\cdots p_{i_a}p_{i_{a+1}}\cdots p_{i_{a+b}}-\phi(p_{i_2}\ldots p_{i_a}p_{i_{a+1}}\cdots p_{i_{a+b}})\right]\\
 & \geq \frac{n}{p_1p_2 \cdots p_r}\times \left[p_1p_2\cdots p_{r-1}-\phi(p_1p_2\cdots p_{r-1})\right].
\end{align*}
Then
\begin{align*}
\kappa(\mathcal{P}(C_n))=|X|& >\phi(n)+|K_1\cup K_3| \\
 & \geq \phi(n) + \frac{n}{p_1p_2 \cdots p_r}\times \left[p_1p_2\cdots p_{r-1}-\phi(p_1p_2\cdots p_{r-1})\right]=\alpha(n),
\end{align*}
a contradiction to (\ref{bound-1}).
\end{proof}

\begin{proposition}\label{one-1}
Let $X$ be a minimum cut-set of $\mathcal{P}(C_n)$ and $A\cup B$ be a separation of $\mathcal{P}(\overline{X})$. If $X$ contains $E_{\frac{n}{p_s}}$ for exactly one $s\in [r]$, then all the remaining sets $E_{\frac{n}{p_i}}$, $i\in [r]\setminus \{s\}$, are contained either in $A$ or in $B$.
\end{proposition}

\begin{proof}
Taking $u=p_1p_2\cdots p_{r-1}-\phi(p_1p_2\cdots p_{r-1})$, we have
\begin{equation}\label{eqn-10}
|X|=\kappa(\mathcal{P}(C_n))\leq \alpha(n)= \phi(n)+\frac{n}{p_1p_2 \cdots p_r}\times u.
\end{equation}
Define the following two sets:
$$ P :=\left\{E_{\frac{n}{p_i}}: i\in [r]\setminus\{s\}, E_{\frac{n}{p_i}}\subseteq A\right\},\;\; Q:=\left\{E_{\frac{n}{p_j}}: j\in [r]\setminus\{s\}, E_{\frac{n}{p_j}}\subseteq B\right\}.$$
Set $a=|P|$ and $b=|Q|$. Then $a+b=r-1$ with $0\leq a\leq r-1$ and $0\leq b\leq r-1$. We show that either $a=0$ or $b=0$. Suppose that $a\geq 1$ and $b\geq 1$.\\

\noindent {\bf Case 1:} $a=1$ or $b=1$.\medskip

Without loss, we may assume that $a=1$. Suppose that $P=\left\{E_{\frac{n}{p_j}}\right\}$ for some $j\in [r]$. Then $Q= \left\{E_{\frac{n}{p_i}}: i\in [r]\setminus \{j, s\}\right\}$.
Since $A\cup B$ is a separation of $\mathcal{P}(\overline{X})$, the subgroups $S_{\frac{n}{p_ip_j}},\; i\in [r]\setminus\{j,s\}$, of $C_n$ are contained in $X$. Let $L$ be the union of these $r-2$ subgroups of $C_n$. By Lemma \ref{none-1}, we have
$$\left\vert L\right\vert =\frac{n}{p_1p_2\cdots p_r}\times \left[\frac{p_1p_2\cdots p_r}{p_j}-p_s\phi\left(\frac{p_1p_2\cdots p_r}{p_sp_j}\right)\right].$$
Since $E_n$, $E_{\frac{n}{p_s}}$ and $L$ are pairwise disjoint and contained in $X$, we get
\begin{align*}
|X| & \geq |E_n|+\left\vert E_{\frac{n}{p_s}}\right\vert + |L| \\
& \geq  \phi(n)+\phi\left(\frac{n}{p_s}\right) +\frac{n}{p_1p_2\cdots p_r}\times \left[\frac{p_1p_2\cdots p_r}{p_j}- p_s\phi\left(\frac{p_1p_2\cdots p_r}{p_sp_j}\right)\right]\\
&\geq \phi(n)+\frac{n}{p_1p_2\cdots p_r}\times\left[ \frac{\phi(p_1p_2\cdots p_r)}{p_s}+ \frac{p_1p_2\cdots p_r}{p_j}- p_s\phi\left(\frac{p_1p_2\cdots p_r}{p_sp_j}\right)\right]\\
 & = \phi(n)+\frac{n}{p_1p_2\cdots p_r}\times v,
\end{align*}
where $v= \frac{\phi(p_1p_2\cdots p_r)}{p_s}+ \frac{p_1p_2\cdots p_r}{p_j}- p_s\phi\left(\frac{p_1p_2\cdots p_r}{p_sp_j}\right)$. We have
$$v-u = \frac{\phi(p_1\cdots p_r)}{p_s} + \frac{p_1\cdots p_r}{p_j}- p_s\phi\left(\frac{p_1\cdots p_r}{p_sp_j}\right) -p_1\ldots p_{r-1}+\phi(p_1\ldots p_{r-1}).$$
If $j=r$, then $v-u = \phi\left(\frac{p_1p_2\cdots p_r}{p_rp_s}\right)\times \left[\frac{\phi(p_rp_s)}{p_s} -1 \right] > 0$. Assume that $s=r$. Then
\begin{align*}
v-u & = \frac{p_1p_2\cdots p_r}{p_rp_j}\times (p_r-p_j) - \phi\left(\frac{p_1p_2\cdots p_r}{p_rp_j}\right)\left[p_r - \frac{\phi(p_j p_r)}{p_r}-\phi(p_j) \right]\\
& = (p_r-p_j)\times \left[ \frac{p_1p_2\cdots p_r}{p_rp_j}-\phi\left(\frac{p_1p_2\cdots p_r}{p_rp_j}\right)\right]+\phi\left(\frac{p_1p_2\cdots p_r}{p_rp_j}\right) \times \left(\frac{\phi(p_j p_r)}{p_r}-1\right).
\end{align*}
If $p_j>2$, then $\frac{\phi(p_j p_r)}{p_r}>1 $ and so $v-u > 0$. Suppose that $p_j=2$. Then $j=1$. In this case,
\begin{align*}
v-u & = (p_r-2)\times \left[p_2\cdots p_{r-1}-\phi\left(p_2\cdots p_{r-1}\right)\right]- \frac{1}{p_r}\phi\left(p_2\cdots p_{r-1}\right)\\
 & \geq (p_r-2)\times \phi(p_3\cdots p_{r-1}) - \frac{1}{p_r}\phi\left(p_2\cdots p_{r-1}\right)\\
 & = \phi(p_3\cdots p_{r-1})\times \left[p_r -2 - \frac{p_2 - 1}{p_r}\right] > 0.
\end{align*}
In the above, $\phi(p_3\cdots p_{r-1})$ is considered to be $1$ if $r=3$. Now assume that $s\neq r$ and $j\neq r$. Then an easy calculation gives that
\begin{align*}
v-u & = \frac{p_1\cdots p_r}{p_rp_j}\times (p_r-p_j)+\phi\left(\frac{p_1\cdots p_r}{p_jp_sp_r}\right)\left[\frac{\phi(p_jp_sp_r)}{p_s}-\phi(p_sp_r)-\phi(p_r)+\phi(p_jp_s)\right]\\
& = \frac{p_1\cdots p_r}{p_rp_j}\times (p_r-p_j)+\phi\left(\frac{p_1\cdots p_r}{p_jp_s}\right)\left[\frac{\phi(p_jp_s)}{p_s}-1\right]-\phi\left(\frac{p_1\cdots p_r}{p_jp_r}\right)(p_r-p_j)\\
& =(p_r-p_j)\times \left[\frac{p_1\cdots p_r}{p_jp_r}-\phi\left(\frac{p_1\cdots p_r}{p_jp_r}\right)\right]+\phi\left(\frac{p_1\cdots p_r}{p_jp_s}\right)\left[\frac{\phi(p_jp_s)}{p_s}-1\right].
\end{align*}
If $p_j>2$, then $\frac{\phi(p_jp_s)}{p_s} \geq 1$ and it follows that $v-u >0$. Suppose that $p_j=2$. Then $j=1$ and $p_s \geq 3$. In this case, for $r\geq 4$,
\begin{align*}
v-u &=(p_r- 2)[p_2\cdots p_{r-1}-\phi(p_2\cdots p_{r-1})]-\frac{1}{p_s}\phi\left(\frac{p_2\cdots p_r}{p_s}\right)\\
& = \frac{1}{p_s}\left[p_s(p_r-2)(p_2\cdots p_{r-1}-\phi(p_2\cdots p_{r-1}))-\phi\left(\frac{p_2\cdots p_r}{p_s}\right)\right]\\
& \geq \frac{1}{p_s}\left[p_s(p_r-2)(p_2\cdots p_{r-1}-\phi(p_2\cdots p_{r-1}))-\phi(p_3\cdots p_r)\right]\\
& > \frac{1}{p_s}\left[\phi(p_r)(p_2\cdots p_{r-1}-\phi(p_2\cdots p_{r-1}))-\phi(p_3\cdots p_r)\right]\\
& = \frac{\phi(p_r)}{p_s}\times [p_2\cdots p_{r-1}-p_2\phi(p_3\cdots p_{r-1})] > 0.
\end{align*}
If $r=3$, then $s=2$ and $v-u=p_3 -2-\frac{\phi(p_3)}{p_2}>0$.
In all the cases, we thus have $v>u$ and hence $|X|\geq \phi(n)+\frac{n}{p_1p_2\cdots p_r}\times v > \phi(n)+\frac{n}{p_1p_2\cdots p_r}\times u$, a contradiction to (\ref{eqn-10}).\\

\noindent {\bf Case 2:} $a\geq 2$ and $b\geq 2$.\medskip

We shall apply a similar argument used in Case 2 of Proposition \ref{all}. Suppose that $P=\{E_{\frac{n}{p_{i_1}}},E_{\frac{n}{p_{i_2}}},\ldots,E_{\frac{n}{p_{i_a}}}\}$ and $Q=\{E_{\frac{n}{p_{i_{a+1}}}},E_{\frac{n}{p_{i_{a+2}}}},\ldots,E_{\frac{n}{p_{i_{a+b}}}}\}$, where
$$[r]\setminus\{s\}=\{i_1,\ldots,i_a,i_{a+1},\ldots, i_{a+b}\}.$$
Without loss, we may assume that $p_{i_1} > p_{i_{a+1}}$. Since $A\cup B$ is a separation of $\mathcal{P}(\overline{X})$, the following subgroups
\begin{align*}
& S_{\frac{n}{p_{i_1}p_{i_{a+1}}}},\; S_{\frac{n}{p_{i_1}p_{i_{a+2}}}},\ldots,\; S_{\frac{n}{p_{i_1}p_{i_{a+b}}}}\\
& S_{\frac{n}{p_{i_2}p_{i_{a+1}}}}, \; S_{\frac{n}{p_{i_2}p_{i_{a+2}}}},\ldots,\; S_{\frac{n}{p_{i_2}p_{i_{a+b}}}}\\
& \vdots \\
& S_{\frac{n}{p_{i_a}p_{i_{a+1}}}},\; S_{\frac{n}{p_{i_a}p_{i_{a+2}}}},\ldots,\; S_{\frac{n}{p_{i_a}p_{i_{a+b}}}}
\end{align*}
of $C_n$ must be contained in $X$. Consider the following three subsets of $C_n$:
$$R_1=\underset{k=1}{\overset{b}{\bigcup}} S_{\frac{n}{p_{i_1}p_{i_{a+k}}}},\;\; R_2=\underset{l=2}{\overset{a}{\bigcup}} S_{\frac{n}{p_{i_l}p_{i_{a+1}}}},\;\; R_3=\underset{l=2}{\overset{a}{\bigcup}} S_{\frac{n}{p_{i_{1}}p_{i_l}}}.$$
Note that $R_1$ and $R_2$ are contained in $X$ but $R_3$ need not be. Also, observe that $R_1\cup R_3$ is the union of the $r-2$ subgroups $S_{\frac{n}{p_{i_1}p_j}}$ with $j\in [r]\setminus\{i_1, s\}$. 

In Case 1, if we suppose that $E_{\frac{n}{p_{i_1}}}$ is contained in $A$ and the sets $E_{\frac{n}{p_{j}}}$, $j\in [r]\setminus\{i_1, s\}$, are contained in $B$, then $X$ will contain the sets $E_n$, $E_{\frac{n}{p_{s}}}$, $R_1\cup R_3$ and it would follow that
\begin{equation}\label{eqn-11}
|E_n|+\left\vert E_{\frac{n}{p_{s}}}\right\vert + |R_1\cup R_3| > \phi(n)+\frac{n}{p_1p_2\cdots p_r}\times u=\alpha(n).
\end{equation}
We shall use this estimate to get a contradiction. By Lemma \ref{none-1},
\begin{align*}
|R_2| & =\frac{n}{p_1p_2 \cdots p_r}\times p_sp_{i_1} p_{i_{a+2}}\cdots p_{i_{a+b}}\left[p_{i_{2}}\cdots p_{i_{a}}-\phi(p_{i_{2}}\cdots p_{i_{a}})\right],\\
|R_3| & =\frac{n}{p_1p_2 \cdots p_r}\times p_s p_{i_{a+1}} p_{i_{a+2}}\cdots p_{i_{a+b}}\left[p_{i_{2}}\cdots p_{i_{a}}-\phi(p_{i_{2}}\cdots p_{i_{a}})\right].
\end{align*}
Since $p_{i_1} > p_{i_{a+1}}$ by our assumption, it follows that $|R_2|>|R_3|$. We have
\begin{align*}
R_1\cap R_2 & =\left(\underset{l=2}{\overset{a}{\bigcup}} S_{\frac{n}{p_{i_1}p_{i_l}p_{i_{a+1}}}}\right)\bigcup \left(\underset{l=2}{\overset{a}{\bigcup}} S_{\frac{n}{p_{i_1}p_{i_l}p_{i_{a+1}}p_{i_{a+2}}}}\right)\bigcup \ldots \bigcup \left(\underset{l=2}{\overset{a}{\bigcup}} S_{\frac{n}{p_{i_1}p_{i_l}p_{i_{a+1}}p_{i_{a+b}}}}\right)\\
 & \subseteq \left(\underset{l=2}{\overset{a}{\bigcup}} S_{\frac{n}{p_{i_1}p_{i_l}p_{i_{a+1}}}}\right)\bigcup \left(\underset{l=2}{\overset{a}{\bigcup}} S_{\frac{n}{p_{i_1}p_{i_l}p_{i_{a+2}}}}\right)\bigcup \ldots \bigcup \left(\underset{l=2}{\overset{a}{\bigcup}} S_{\frac{n}{p_{i_1}p_{i_l}p_{i_{a+b}}}}\right)\\
& = R_1\cap R_3.
\end{align*}
Thus $|R_1\cap R_2|\leq |R_1\cap R_3|$ and hence
$$|R_1\cup R_2|=|R_1|+|R_2|-|R_1\cap R_2|> |R_1|+|R_3|-|R_1\cap R_3|= |R_1\cup R_3|.$$
Then $|X| \geq |E_n|+\left\vert E_{\frac{n}{p_s}} \right\vert + |R_1\cup R_2|> \phi(n)+ \phi\left(\frac{n}{p_s}\right) + |R_1\cup R_3|> \alpha(n)$ using (\ref{eqn-11}), which is a contradiction to (\ref{eqn-10}).
\end{proof}

\begin{corollary}\label{one-2}
Let $X$ be a minimum cut-set of $\mathcal{P}(C_n)$. If $X$ contains $E_{\frac{n}{p_s}}$ for exactly one $s\in [r]$, then $n_s\geq 2$.
\end{corollary}

\begin{proof}
Let $A\cup B$ be a separation  of $\mathcal{P}(\overline{X})$. By Proposition \ref{one-1}, we may assume that the sets $E_{\frac{n}{p_i}}$, $i\in [r]\setminus \{s\}$, are contained in $A$. Then an arbitrary element in $B$ must have order of the form $\frac{n}{p_s^k}$, where $0\leq k\leq n_s$. The sets $E_n$ and $E_{\frac{n}{p_s}}$ are contained in $X$ and the order of the elements in these two sets correspond to $k=0,1$ respectively. Since $B$ is non-empty, we must have that $n_s\geq 2$.
\end{proof}

\begin{theorem}\label{one-6}
Let $X$ be a minimum cut-set of $\mathcal{P}(C_n)$ containing $E_{\frac{n}{p_s}}$ for a unique $s\in [r]$. Then $s=r$ and $\kappa(\mathcal{P}(C_n))=\beta_r(n)=\beta(n)$.
\end{theorem}

\begin{proof}
Recall that $\beta_s(n)=\phi(n) + \frac{n}{p_1\cdots p_r}\times \frac{1}{p_s^{n_s-1}}\left[\frac{p_1p_2\cdots p_{r}}{p_s} +\phi\left(\frac{p_1p_2\cdots p_{r}}{p_s}\right)\left( p_s^{n_s-1}-2\right) \right]$ and that $|X|= \kappa(\mathcal{P}(C_n))\leq \min \{\alpha(n), \beta_s(n)\}$.

We have $n_s\geq 2$ by Corollary \ref{one-2}. Fix a separation $A\cup B$ of $\mathcal{P}(\overline{X})$. By Proposition \ref{one-1}, we may assume  that all the sets $E_{\frac{n}{p_i}}$, $i\in [r]\setminus\{s\}$, are  contained in $A$. Then any element in $B$ must be of order of the form
$$p_1^{n_1}\cdots p_{s-1}^{n_{s-1}}p_s^tp_{s+1}^{n_{s+1}}\cdots p_r^{n_r}$$
for some $t$ with $0\leq t \leq n_s-2$, as $E_n$ and $E_{\frac{n}{p_s}}$ are contained in $X$. Let $k\in \{0,1,\cdots,n_s -2\}$ be the largest integer for which $B$ has an element of order $p_1^{n_1}\cdots p_{s-1}^{n_{s-1}}p_s^k p_{s+1}^{n_{s+1}}\cdots p_r^{n_r}$. Then the elements of $C_n$ of order $p_1^{n_1}\cdots p_{s-1}^{n_{s-1}}p_s^jp_{s+1}^{n_{s+1}}\cdots p_r^{n_r}$ with $k+1 \leq j\leq n_s$ are in $X$.

Thus the sets $E_n$, $E_{\frac{n}{p_s^{l}}}$ with $1\leq l\leq n_s-k-1$ and the subgroups $S_{\frac{n}{p_s^{n_s -k}p_i}},\;i\in [r]\setminus\{ s \}$, are contained in $X$. We have
\begin{align*}
\left\vert \underset{l=1}{\overset{n_s-k-1}{\bigcup}} E_{\frac{n}{p_s^{l}}}\right\vert = \underset{l=1}{\overset{n_s-k-1}{\sum}} \left\vert  E_{\frac{n}{p_s^{l}}}\right\vert &=\phi\left(\frac{n}{p_s^{n_s}} \right)\left[\phi(p_s^{n_s -1})+\cdots +\phi(p_s^{k+1})\right]\\
& =\phi\left(\frac{n}{p_s^{n_s}} \right)\left[p_s^{n_s -1} - p_s^k \right]\\
& = \frac{n}{p_1\cdots p_r}\times \frac{1}{p_s^{n_s-1}}\times \phi\left(\frac{p_1p_2\cdots p_{r}}{p_s}\right)\left[p_s^{n_s -1} - p_s^k \right].
\end{align*}
Applying a similar argument as in the proof Lemma \ref{none-1}, it can be calculated that
$$\left\vert \underset{i\neq s}{\underset{i=1}{\overset{r}{\bigcup}}} S_{\frac{n}{p_s^{n_s -k}p_i}}\right\vert =\frac{n}{p_1\cdots p_r}\times \frac{1}{p_s^{n_s -k-1}}\times\left[ \frac{p_1\cdots p_r}{p_s}-\phi\left(\frac{p_1\cdots p_r}{p_s}\right)\right]$$
Therefore, $|X|\geq |E_n|+\left\vert \underset{l=1}{\overset{n_s-k-1}{\bigcup}} E_{\frac{n}{p_s^{l}}}\right\vert + \left\vert \underset{i\neq s}{\underset{i=1}{\overset{r}{\bigcup}}} S_{\frac{n}{p_s^{n_s -k}p_i}}\right\vert =\mu_k$, where
$$\mu_k= \phi(n)+ \frac{n}{p_1\cdots p_r}\times \left[\phi\left(\frac{p_1\cdots p_r}{p_s}\right)+\frac{1}{p_s^{n_s -k-1}}\left[\frac{p_1\cdots p_r}{p_s}-2\phi\left(\frac{p_1\cdots p_r}{p_s}\right)\right]\right].$$
Since $r\geq 3$, $\frac{p_1\cdots p_r}{p_s}\neq 2\phi\left(\frac{p_1\cdots p_r}{p_s}\right)$, see \cite[Theorem 1.3(iii)]{cps}. If $\frac{p_1\cdots p_r}{p_s}< 2\phi\left(\frac{p_1\cdots p_r}{p_s}\right)$, then $\mu_k$ is minimum when $k=n_s -2$. In that case,
\begin{align*}
|X|\geq \mu_{n_s -2} &= \phi(n)+\frac{n}{p_1\cdots p_r}\times\left[\phi\left(\frac{p_1\cdots p_r}{p_s}\right)+\frac{1}{p_s}\left[\frac{p_1\cdots p_r}{p_s}-2\phi\left(\frac{p_1\cdots p_r}{p_s}\right)\right]\right]\\
 & > \phi(n)+\frac{n}{p_1\cdots p_r}\times\left[\phi\left(\frac{p_1\cdots p_r}{p_s}\right)+\frac{p_1\cdots p_r}{p_s}-2\phi\left(\frac{p_1\cdots p_r}{p_s}\right)\right]\\
 & = \phi(n)+\frac{n}{p_1\cdots p_r}\times\left[\frac{p_1\cdots p_r}{p_s}-\phi\left(\frac{p_1\cdots p_r}{p_s}\right)\right]\\
& \geq \phi(n)+\frac{n}{p_1\cdots p_r}\times [p_1\cdots p_{r-1}-\phi(p_1\cdots p_{r-1}]=\alpha(n),
\end{align*}
which is a contradiction to that $|X|\leq \alpha(n)$. In the above, the second last inequality follows from Lemma \ref{inequ-3}.\medskip

\noindent Thus $\frac{p_1\cdots p_r}{p_s}> 2\phi\left(\frac{p_1\cdots p_r}{p_s}\right)$. In this case, $\mu_k$ is minimum when $k=0$ and so
\begin{align*}
|X|\geq \mu_0 &= \phi(n)+\frac{n}{p_1\cdots p_r}\times\left[\phi\left(\frac{p_1\cdots p_r}{p_s}\right)+\frac{1}{p_s^{n_s -1}}\left[\frac{p_1\cdots p_r}{p_s}-2\phi\left(\frac{p_1\cdots p_r}{p_s}\right)\right]\right]\\
 & = \phi(n)+\frac{n}{p_1\cdots p_r}\times\frac{1}{p_s^{n_s -1}}\times\left[\frac{p_1\cdots p_r}{p_s}+\phi\left(\frac{p_1\cdots p_r}{p_s}\right)\left(p_s^{n_s -1} -2 \right)\right]\\
 & =\beta_s(n).
\end{align*}
Since $|X|\leq \beta_s(n)$, it follows that $|X|=\beta_s(n)$.

We now claim that $s=r$. Suppose that $s\neq r$. Then $s< r$. Since $\frac{p_1\cdots p_r}{p_s}> 2\phi\left(\frac{p_1\cdots p_r}{p_s}\right)$, Lemma \ref{compa-3} implies that $\beta_s(n) >\beta_r(n)$. Then $\kappa(\mathcal{P}(C_n))=|X|=\beta_s(n) >\beta_r(n)=\beta(n)$, a contradiction to (\ref{bound-2}).
\end{proof}

\begin{proposition}\label{one-atmost}
If $n_r \geq 2$ and $X$ is a minimum cut-set of $\mathcal{P}(C_n)$, then $X$ contains at most one of the sets $E_{\frac{n}{p_1}}, E_{\frac{n}{p_2}},\dots, E_{\frac{n}{p_r}}$.
\end{proposition}

\begin{proof}
If possible, suppose that the sets $E_{\frac{n}{p_k}}$ and $E_{\frac{n}{p_l}}$ are contained in $X$ for some $k,l\in [r]$ with $k\neq l$. We may assume that $k<l$. Then $k\leq r-1$. Using Lemmas \ref{inequ-1-1} and \ref{inequ-2},
$$\left\vert E_{\frac{n}{p_l}} \right\vert = \phi\left(\frac{n}{p_l}\right)\geq \phi\left(\frac{n}{p_r}\right)\geq \frac{n}{p_1\cdots p_r}\times \frac{\phi\left(p_1\cdots p_{r-1}p_{r} \right)}{p_r}$$
and
$$\left\vert E_{\frac{n}{p_k}} \right\vert = \phi\left(\frac{n}{p_k}\right)\geq p_{r}^{n_{r}-1} \phi\left(\frac{n}{p_{r}^{n_{r}}}\right)=\frac{n}{p_1\cdots p_r}\times \phi\left(p_1\cdots p_{r-2}p_{r-1} \right).$$
We have $|X|=\kappa(\mathcal{P}(C_n))\leq \beta(n)=\phi(n)+ \frac{n}{p_1\cdots p_r}\times u$, where
$$u= \frac{1}{p_r^{n_r-1}}\left[p_1p_2\cdots p_{r-1} +\phi(p_1\cdots p_{r-1})\left( p_r^{n_r-1}-2\right) \right].$$
Using Lemma \ref{inequ-1} and the hypothesis that $n_r\geq 2$, an easy calculation gives that
\begin{align*}
\phi\left(p_1\cdots p_{r-1}\right)+ \frac{\phi\left(p_1\cdots p_{r} \right)}{p_r} - u & = \frac{1}{p_r^{n_r-1}}\left[\left(\phi\left(p_r^{n_r -1}\right)+2\right)\phi\left(p_1\cdots p_{r-1}\right) - p_1\cdots p_{r-1} \right]\\
 & > \frac{1}{p_r^{n_r-1}}\left[\left(r+1\right)\phi\left(p_1\cdots p_{r-1}\right) - p_1\cdots p_{r-1} \right]\\
 & > 0.
\end{align*}
Therefore $\phi\left(p_1\cdots p_{r-1}\right)+ \frac{\phi\left(p_1\cdots p_{r} \right)}{p_r} > u$. Since $E_n, E_{\frac{n}{p_k}}, E_{\frac{n}{p_l}}$ are pairwise disjoint and contained in $X$, we get
\begin{align*}
|X| & \geq \left\vert E_n\right\vert + \left\vert E_{\frac{n}{p_k}}\right\vert + \left\vert E_{\frac{n}{p_l}}\right\vert\\
 & = \phi(n)+ \frac{n}{p_1\cdots p_r}\times\left[\phi\left(p_1\cdots p_{r-1}\right)+ \frac{\phi\left(p_1\cdots p_{r} \right)}{p_r} \right]\\
 & > \phi(n)+ \frac{n}{p_1\cdots p_r}\times u= \beta(n),
\end{align*}
a contradiction to that $|X|=\kappa(\mathcal{P}(C_n))\leq \beta(n)$.
\end{proof}

\begin{proof}[{\bf Proof of Theorem \ref{res-2}}]
Since $n_r\geq 2$, $X$ contains at most one of the sets $E_{\frac{n}{p_1}},\dots, E_{\frac{n}{p_r}}$ by Proposition \ref{one-atmost}. If $X$ does not contain any of these sets, then Proposition \ref{all} implies that $\kappa(\mathcal{P}(C_n))=\alpha(n)$. If $X$ contains $E_{\frac{n}{p_s}}$ for a unique $s\in [r]$, then Proposition \ref{one-6} gives that $s=r$ and $\kappa(\mathcal{P}(C_n))=\beta_r(n)=\beta(n)$. It follows that $\kappa(\mathcal{P}(C_n)) = \min \{\alpha(n), \beta(n)\}$, thus proving the theorem.
\end{proof}

\section{Proof of Theorem \ref{res-3}}

Let $n=p_1p_2\cdots p_r$, where $r\geq 3$ and $p_1,p_2,\ldots,p_r$ are prime numbers with $p_1<p_2<\cdots <p_r$. In this case,
\begin{enumerate}
\item[] $\alpha(n)= \phi(n) + p_1p_2\cdots p_{r-1}-\phi(p_1p_2\cdots p_{r-1})$,
\item[] $\gamma(n)= \phi(n)+\phi(p_1\cdots p_{r-1})+ \phi(p_1\cdots p_{r-2}p_r)+p_1\cdots p_{r-2}-\phi(p_1\cdots p_{r-2})$.
\end{enumerate}
Recall that the set $X_{r-1,r} =E_n\bigcup E_{\frac{n}{p_{r-1}}}\bigcup E_{\frac{n}{p_{r}}}\bigcup \left(\underset{i=1}{\overset{r-2}\bigcup} S_{\frac{n}{p_ip_{r-1}p_r}}\right)$ defined in Section \ref{new-bound} is a cut-set of $\mathcal{P}(C_n)$ which gives the upper bound $\gamma_{r-1,r}(n)=\gamma(n)$ for $\kappa(\mathcal{P}(C_n))$. 
Now, let $X$ be a minimum cut-set of $\mathcal{P}(C_n)$. Then $|X|=\kappa(\mathcal{P}(C_n))\leq \min\{\alpha(n),\gamma(n)\}$. 

\begin{proposition}\label{two-0}
$X$ contains none or exactly two of the sets $E_{\frac{n}{p_1}}, E_{\frac{n}{p_2}},\dots, E_{\frac{n}{p_r}}$.
\end{proposition}

\begin{proof}
This follows from Proposition \ref{two-atmost} and Corollary \ref{one-2}.
\end{proof}

\begin{proposition}\label{two-1}
If $X$ contains none of the sets $E_{\frac{n}{p_1}},\dots, E_{\frac{n}{p_r}}$, then $\kappa(\mathcal{P}(C_n))=\alpha(n)$.
\end{proposition}

\begin{proof}
This follows from Proposition \ref{all}.
\end{proof}

\begin{proposition}\label{two-2}
If $X$ contains exactly two of the sets $E_{\frac{n}{p_1}}$, $E_{\frac{n}{p_2}},\dots, E_{\frac{n}{p_r}}$, then the following hold:
\begin{enumerate}
\item[(i)] $E_{\frac{n}{p_{r-1}}}$ and $E_{\frac{n}{p_r}}$ are contained in $X$.

\item[(ii)] $X=X_{r-1,r}$ and so $\kappa(\mathcal{P}(C_n))=\gamma_{r-1,r}(n)=\gamma(n)$.
\end{enumerate}
\end{proposition}

\begin{proof}
(i) Let $E_{\frac{n}{p_j}}$ and $E_{\frac{n}{p_k}}$ be the two sets which are contained in $X$, where $1\leq j\neq k\leq r$.
If possible, suppose that $\{j,k\}\neq \{r-1,r\}$. Without loss, we may assume that $j < k$. Then $j\leq r-2$. By Lemma \ref{inequ-1-1},
\begin{equation}\label{eqn-5}
\left\vert E_{\frac{n}{p_j}}\right\vert = \phi\left(\frac{n}{p_j}\right)> \phi\left(\frac{n}{p_{r-1}}\right)
\end{equation}
and
\begin{equation}\label{eqn-6}
\left\vert E_{\frac{n}{p_k}}\right\vert = \phi\left(\frac{n}{p_k}\right)\geq \phi\left(\frac{n}{p_{r}}\right)> \phi\left(\frac{n}{p_{r}}\right) - \phi\left(\frac{n}{p_{r-1}p_{r}}\right).
\end{equation}
Now fix a separation $A\cup B$ of $\mathcal{P}(\overline{X})$. We consider two cases.\\

\noindent {\bf Case 1:} Each of $A$ and $B$ contains at least one of the sets $E_{\frac{n}{p_i}}$, $i\in [r]\setminus \{j,k\}$.\medskip

This possibility occurs only when $r\geq 4$. Suppose that $A$ contains $E_{\frac{n}{p_s}}$ and $B$ contains $E_{\frac{n}{p_t}}$ for some $s,t\in [r]\setminus \{j,k\}$ with $s\neq t$. Then the subgroup $S_{\frac{n}{p_s p_t}}$ is contained in $X$. We have
\begin{equation}\label{eqn-7}
\left\vert S_{\frac{n}{p_s p_t}}\right\vert = \frac{n}{p_sp_t} \geq  \frac{n}{p_{r-1}p_r}.
\end{equation}
Since the sets $E_n, E_{\frac{n}{p_j}}$, $E_{\frac{n}{p_k}}$ and $S_{\frac{n}{p_s p_t}}$ are pairwise disjoint and contained in $X$, we get using (\ref{eqn-5}), (\ref{eqn-6}) and (\ref{eqn-7}) that
\begin{align*}
|X| & \geq \left\vert E_n\right\vert+\left\vert E_{\frac{n}{p_j}}\right\vert +\left\vert E_{\frac{n}{p_k}}\right\vert +\left\vert S_{\frac{n}{p_s p_t}} \right\vert \\
 & > \phi(n)+\phi\left(\frac{n}{p_r}\right) + \phi\left(\frac{n}{p_{r-1}}\right)+ \frac{n}{p_{r-1}p_r} -\phi\left( \frac{n}{p_{r-1}p_r}\right) = \gamma(n),
\end{align*}
which is a contradiction to that $|X|\leq \gamma(n)$.\\

\noindent {\bf Case 2:} All the sets $E_{\frac{n}{p_i}}$, $i\in [r]\setminus \{j,k\}$, are contained either in $A$ or in $B$.\medskip

Without loss, we may assume that $B$ contains all the sets $E_{\frac{n}{p_i}}$, $i\in [r]\setminus \{j,k\}$. Since $A\cup B$ is a separation of $\mathcal{P}(\overline{X})$, the order of each element of $A$ must be divisible by $p_i$ for every $i\in [r]\setminus \{j,k\}$. Since $E_n, E_{\frac{n}{p_j}}$ and $E_{\frac{n}{p_k}}$ are contained in $X$, it follows that the order of every element of $A$ is equal to $\frac{n}{p_j p_k}$, thus giving $A=E_{\frac{n}{p_j p_k}}$. Then each of the subgroups $S_{\frac{n}{p_ip_jp_k}}$, $i\in [r]\setminus \{j,k\}$, of $C_n$ is contained in $X$. Let $R$ be the union of these $r-2$ subgroups.
By Lemmas \ref{none-1} and \ref{inequ-3}, we get
\begin{equation}\label{eqn-8}
|R| = \frac{n}{p_j p_k}-\phi\left(\frac{n}{p_j p_k}\right)> \frac{n}{p_{r-1} p_r}-\phi\left(\frac{n}{p_{r-1} p_r}\right),
\end{equation}
as $\{j,k\}\neq \{r-1,r\}$. Since the sets $E_n, E_{\frac{n}{p_j}}$, $E_{\frac{n}{p_k}}$ and $R$ are pairwise disjoint and contained in $X$, we get using (\ref{eqn-5}), (\ref{eqn-6}) and (\ref{eqn-8}) that
\begin{align*}
|X| & \geq \left\vert E_n\right\vert+\left\vert E_{\frac{n}{p_j}}\right\vert +\left\vert E_{\frac{n}{p_k}}\right\vert +\left\vert R \right\vert \\
 & > \phi(n)+\phi\left(\frac{n}{p_r}\right) + \phi\left(\frac{n}{p_{r-1}}\right)+ \frac{n}{p_{r-1}p_r} -\phi\left( \frac{n}{p_{r-1}p_r}\right) = \gamma(n),
\end{align*}
which is a contradiction to that $|X|\leq \gamma(n)$. This completes the proof of (i).\\

(ii) By (i), $E_{\frac{n}{p_{r-1}}}$ and $E_{\frac{n}{p_r}}$ are contained in $X$. Fix a separation $A\cup B$ of $\mathcal{P}(\overline{X})$. If $E_{\frac{n}{p_s}}\subseteq A$ and $E_{\frac{n}{p_t}}\subseteq B$ for some $s,t\in [r-2]$ with $s\neq t$ (possible only when $r\geq 4$), then
the subgroup $S_{\frac{n}{p_s p_t}}$ is contained in $X$. We have
\begin{equation}\label{eqn-9}
\left\vert S_{\frac{n}{p_s p_t}}\right\vert = \frac{n}{p_sp_t} >  \frac{n}{p_{r-1}p_r}> \frac{n}{p_{r-1}p_r} -\phi\left( \frac{n}{p_{r-1}p_r}\right).
\end{equation}
Since the sets $E_n, E_{\frac{n}{p_r}}, E_{\frac{n}{p_{r-1}}}$ and $S_{\frac{n}{p_s p_t}}$ are pairwise disjoint and contained in $X$, we get using (\ref{eqn-9}) that
\begin{align*}
|X| & \geq \left\vert E_n\right\vert +\left\vert E_{\frac{n}{p_r}}\right\vert +\left\vert E_{\frac{n}{p_{r-1}}}\right\vert +\left\vert S_{\frac{n}{p_s p_t}} \right\vert \\
 & > \phi(n)+\phi\left(\frac{n}{p_r}\right) + \phi\left(\frac{n}{p_{r-1}}\right)+ \frac{n}{p_{r-1}p_r} -\phi\left( \frac{n}{p_{r-1}p_r}\right) = \gamma(n),
\end{align*}
which is a contradiction to that $|X|\leq \gamma(n)$. Therefore, all the sets $E_{\frac{n}{p_i}}$, $1\leq i\leq r-2$, are contained either in $A$ or in $B$.

Without loss, we may assume that all these $r-2$ sets are contained in $B$. Applying the argument as in the proof of Case 2 of (i), we get that $A=E_{\frac{n}{p_{r-1}p_r}}$ and that each of the subgroups $S_{\frac{n}{p_ip_{r-1}p_r}}$, $i\in [r-2]$, is contained in $X$. Thus $X$ contains the cut-set $X_{r-1,r}$. Since $X$ is a minimum cut-set of $\mathcal{P}(C_n)$, it follows that $X=X_{r-1,r}$ and hence $\kappa(\mathcal{P}(C_n))= |X|=\left\vert X_{r-1,r}\right\vert=\gamma(n)$.
\end{proof}

Now, we can see that Theorem \ref{res-3} follows from Propositions \ref{two-0}, \ref{two-1} and \ref{two-2}. Note that the comparison between $\alpha(n)$ and $\gamma(n)$ is given in Lemma \ref{compa-4}(ii).

\vskip .5cm

\noindent{\bf Addresses}:\\

\noindent Sriparna Chattopadhyay, Kamal Lochan Patra, Binod Kumar Sahoo\\

\noindent 1) School of Mathematical Sciences\\
National Institute of Science Education and Research (NISER), Bhubaneswar\\
P.O.- Jatni, District- Khurda, Odisha - 752050, India\medskip

\noindent 2) Homi Bhabha National Institute (HBNI)\\
Training School Complex, Anushakti Nagar\\
Mumbai - 400094, India\medskip

\noindent E-mails: sriparna@niser.ac.in, klpatra@niser.ac.in, bksahoo@niser.ac.in\\

\begin{thebibliography}{99}
\bibitem{survey} J. Abawajy, A. Kelarev and M. Chowdhury, Power graphs: a survey, Electron. J. Graph Theory Appl. 1 (2013), 125--147.
\bibitem{BIS} D. Bubboloni, Mohammad A. Iranmanesh and S. M. Shaker, On some graphs associated with the finite alternating groups, Comm. Algebra 45 (2017), 5355--5373.
\bibitem{BIS-1} D. Bubboloni, Mohammad A. Iranmanesh, and S. M. Shaker, Quotient graphs for power graphs, Rend. Semin. Mat. Univ. Padova 138 (2017), 61--89.
\bibitem{cam} P. J. Cameron, The power graph of a finite group, II, J. Group Theory 13 (2010), 779--783.
\bibitem{cam-1} P. J. Cameron and S. Ghosh, The power graph of a finite group, Discrete Math. 311 (2011), 1220--1222.
\bibitem{ivy} I. Chakrabarty, S. Ghosh and M. K. Sen, Undirected power graphs of semigroups, Semigroup Forum 78 (2009), 410--426.
\bibitem{sri} S. Chattopadhyay and P. Panigrahi, Connectivity and planarity of power graphs of finite cyclic, dihedral and dicyclic groups, Algebra Discrete Math. 18 (2014), 42--49.
\bibitem{CP-2015} S. Chattopadhyay and P. Panigrahi, On Laplacian spectrum of power graphs of finite cyclic and dihedral groups, Linear Multilinear Algebra 63 (2015), 1345--1355.
\bibitem{cps} S. Chattopadhyay, K. L. Patra and B. K. Sahoo, Vertex connectivity of the power graph of a finite cyclic group, Discrete Appl. Math., to appear, https://doi.org/10.1016/j.dam.2018.06.001.
\bibitem{cps-1} S. Chattopadhyay, K. L. Patra and B. K. Sahoo, Minimal cut-sets in the power graph of certain finite non-cyclic groups, communicated, https://arxiv.org/abs/1802.07646.
\bibitem{cur} B. Curtin and G. R. Pourgholi, Edge-maximality of power graphs of finite cyclic groups, J. Algebraic Combin. 40 (2014), 313--330.
\bibitem{cur-1} B. Curtin and G. R. Pourgholi, An Euler totient sum inequality, J. Number Theory 163 (2016), 101--113.
\bibitem{dooser} A. Doostabadi, A. Erfanian and A. Jafarzadeh, Some results on the power graphs of finite groups, ScienceAsia 41 (2015), 73--78.
\bibitem{doos} A. Doostabadi and M. Farrokhi D. Ghouchan, On the connectivity of proper power graphs of finite groups, Comm. Algebra 43 (2015), 4305--4319.
\bibitem{FMW} M. Feng, X. Ma and K. Wang, The structure and metric dimension of the power graph of a finite group, European J. Combin. 43 (2015), 82--97.
\bibitem{FMW-1} M. Feng, X. Ma and K. Wang, The full automorphism group of the power (di)graph of a finite group, European J. Combin. 52 (2016), 197--206.
\bibitem{ker1} A. V. Kelarev and S. J. Quinn, A combinatorial property and power graphs of groups, in Contributions to general algebra, 12 (Vienna, 1999), 229--235, Heyn, Klagenfurt, 2000.
\bibitem{ker1-1} A. V. Kelarev, S. J. Quinn and R. Smol\'{\i}kov\'{a}, Power graphs and semigroups of matrices, Bull. Austral. Math. Soc. 63 (2001), 341--344.
\bibitem{ker2} A. V. Kelarev and S. J. Quinn, Directed graphs and combinatorial properties of semigroups, J. Algebra 251 (2002), 16--26.
\bibitem{mir} M. Mirzargar, A. R. Ashrafi and M. J. Nadjafi-Arani, On the power graph of a finite group, Filomat 26 (2012), 1201--1208.
\bibitem{mog} A. R. Moghaddamfar, S. Rahbariyan and W. J. Shi, Certain properties of the power graph associated with a finite group, J. Algebra Appl. 13 (2014), no. 7, 1450040, 18 pp.
\bibitem{panda} R. P. Panda and K. V. Krishna, On connectedness of power graphs of finite groups, J. Algebra Appl. 17 (2018), no. 10, 1850184, 20 pp.
\end{thebibliography}
\end{document}